\documentclass[11pt,reqno]{amsart}
\usepackage[margin=1.5in]{geometry}

\usepackage[utf8]{inputenc} 

\usepackage[doi=false, url=false, eprint=false, maxbibnames=10, firstinits=true]{biblatex}
\usepackage{amsfonts,bm,mathrsfs}
\usepackage{amsmath,amscd,textcomp,amssymb, mathtools}
\usepackage{nccmath}
\usepackage{enumitem}
\usepackage{moreenum}
\usepackage{graphicx}
\usepackage[all]{xy}
\usepackage{todonotes}
\usepackage{tikz}

\usepackage{etoolbox}
\usepackage{mathtools}
\usepackage{xcolor}
\definecolor{cite}{HTML}{EBCB8B}
\definecolor{link}{HTML}{B48EAD}
\definecolor{url}{HTML}{A3BE8C}
\usepackage{hyperref}
\hypersetup{
  linkcolor  = link!85!black,
  citecolor  = cite!85!black,
  urlcolor   = url!85!black,
  colorlinks = true,
  }
\usepackage{cleveref}

\newtheorem{theorem}{Theorem}[section]
\newtheorem{prop}[theorem]{Proposition}
\newtheorem{lem}[theorem]{Lemma}
\newtheorem{cor}[theorem]{Corollary}

\theoremstyle{definition}
\newtheorem{definition}[theorem]{Definition}

\theoremstyle{remark}
\newtheorem{remark}[theorem]{Remark}
\numberwithin{equation}{section}

\makeatletter
\patchcmd{\colon}{{:}}{\mathopen:}{}{}
\makeatother

\makeatletter
\def\namedlabel#1#2{\begingroup
   \def\@currentlabel{#2}%
   \label{#1}\endgroup
}
\makeatother

\newcommand{\verteq}{\rotatebox{90}{$\,=$}}
\newcommand{\equalto}[2]{\underset{\scriptstyle\overset{\mkern4mu\verteq}{#2}}{#1}}

\newcommand{\R}{{\mathbb R}}
\newcommand{\Z}{{\mathbb Z}}
\newcommand{\N}{{\mathbb N}}
\newcommand{\scrJ}{{\mathcal J}}

\newcommand{\cA}{{\mathcal{A}}}

\newcommand{\cM}{{\mathcal{M}}}

\newcommand{\y}{{\mathfrak y}}

\newcommand{\CZ}{\mu_{\operatorname{CZ}}}
\newcommand{\tCZ}{\mu^{\operatorname{tr}}_{\operatorname{CZ}}}

\newcommand{\Crit}{{\operatorname{Crit}}}
\newcommand{\ev}{{\operatorname{ev}}}

\makeatletter
\def\namedlabel#1#2{\begingroup
   \def\@currentlabel{#2}%
   \label{#1}\endgroup
}
\makeatother

\begin{document}

\title[Computing RFH of tentacular hyperboloids]{Computing the Rabinowitz Floer homology of tentacular hyperboloids}


\author{A. Fauck}
\address{Department of Mathematics, Humboldt-Universit\"{a}t zu Berlin,Germany.}
\curraddr{}
\email{fauck@math.hu-berlin.de}
\thanks{}

\author{W. J. Merry}
\address{Department of Mathematics, ETH Z\"{u}rich, Switzerland.}
\curraddr{}
\email{merry@math.ethz.ch}
\thanks{}

\author{J. Wi\'{s}niewska}
\address{Department of Mathematics, ETH Z\"{u}rich, Switzerland.}
\curraddr{}
\email{jagna.wisniewska@math.ethz.ch}
\thanks{}



\date{\today}

\dedicatory{}

\begin{abstract}
We compute the Rabinowitz Floer homology for a class of non-compact hyperboloids $\Sigma\simeq S^{n+k-1}\times\mathbb{R}^{n-k}$. Using an embedding of a compact sphere $\Sigma_0\simeq S^{2k-1}$ into the hypersurface $\Sigma$, we construct a chain map from the Floer complex of $\Sigma$ to the Floer complex of $\Sigma_0$.
In contrast to the compact case, the Rabinowitz Floer homology groups of $\Sigma$ are both non-zero and not equal to its singular homology. As a consequence, we deduce that the Weinstein Conjecture holds for any strongly tentacular deformation of such a hyperboloid.
\end{abstract}
\maketitle

\section{Introduction}
Rabinowitz Floer homology is a homology theory for exact contact hypersurfaces in symplectic manifolds. It has various applications in symplectic and contact geometry: it provides obstructions for exact embeddings of contact manifolds, it can be used to distinguish contact structures, and it gives qualitative information on the Reeb flow on a contact manifold. Rabinowitz Floer homology was originally defined by Cieliebak and Frauenfelder in \cite{CieliebakFrauenfelder2009} for compact contact type hypersurfaces of exact convex symplectic manifolds, and has since been extended to cover more general compact hypersurfaces in more general symplectic manifolds. These include certain stable Hamiltonian structures \cite{CieliebakFrauenfelderPaternain2010}, negative line bundles \cite{AlbersKang2016}, and symplectisations of hypertight contact manifolds \cite{AlbersFuchsMerry2015}.  In a different direction, and more relevantly for the present article, the third author constructed in her PhD thesis \cite{Wisniewska2017} a class of so-called \textit{tentacular} Hamiltonians on $ \mathbb{R}^{2n}$ with \textit{non-compact} level sets for which Rabinowitz Floer homology is well defined. 

In the compact case, Rabinowitz Floer homology is eminently computable. For example, Rabinowitz Floer homology has been completely computed for unit cotangent bundles \cite{CieliebakFrauenfelderOancea2010,Abbon2009}, magnetic cotangent bundles \cite{Merry2011a}, Brieskorn spheres \cite{Fauck2015}, negative line bundles and certain annulus subbundles thereof \cite{AlbersKang2016,Venkatesh}. Many of these computations either rely on -- or can alternatively be proved by -- the intimate relationship of Rabinowitz Floer homology with symplectic homology \cite{CieliebakFrauenfelderOancea2010,CieliebakOancea2018}. 

The non-compact case is rather less tractable. In \cite{pasquotto2017,pasquotto2018} the class of non-compact hypersurfaces for which the Rabinowitz Floer homology could be defined was extended, and a rather general invariance result  was proved. However no explicit computations were presented. The goal of the present article is to remedy this deficit, by providing a complete calculation of the Rabinowitz Floer homology for a class of (deformations of) symplectic hyperboloids. As a byproduct, using the aforementioned invariance result we establish the validity of the Weinstein Conjecture for all such non-compact hypersurfaces. 

Here are the details. Consider a  quadratic Hamiltonian $H$ on $T^*\mathbb{R}^n$ of the form $H(z)\coloneqq \frac{1}{2}z^TAz -1$, where $A$ is a non-degenerate symmetric matrix. The H\"ormander classification of symplectic forms \cite{Hormander1995} tells us that by performing a linear symplectic change of coordinates, $H$ can be brought into a certain standard form, and moreover that these standard forms are classified by the eigenvalues of the Jordan decomposition\footnote{Here $\mathbb{J}$ stands for the standard complex structure on $\mathbb{C}^n \simeq T^*\R^n$.} of $\mathbb{J}A$. We are interested in the case where this linear symplectic change of coordinates yields a decomposition
\begin{equation}
    \label{defH}
    H(x,y) = \underset{\coloneqq H_0(x)}{\underbrace{\frac{1}{2}x^TA_0 x  - 1}} + \underset{ \coloneqq H_1(y)}{\underbrace{\frac{1}{2}y^TA_1 y}},
\end{equation}
for $(x,y) \in T^\ast\mathbb{R}^k\times T^\ast\mathbb{R}^{n-k}$, where $A_0$ is positive definite and  $\mathbb{J}A_1$ is hyperbolic. Here and throughout this paper, we implicitly assume that $1 \le k \le n-1$ (see Remark \ref{k0n} for the cases $k=0,n$). This implies that $A_1$ has signature $(n-k,n-k)$. If we denote the regular level sets 
\begin{equation}
    \label{Sigma-Sigma0}
    \Sigma\coloneqq H^{-1}(0)\qquad\textrm{and}\qquad \Sigma_0\coloneqq H_0^{-1}(0),
\end{equation}
then $\Sigma$ is a hyperboloid diffeomorphic to $S^{n+k-1}\times\R^{n-k}$, whereas $\Sigma_0$ is an ellipsoid diffeomorphic to $S^{2k-1}$. Note that $ \Sigma$ is never compact. 

We say that $H$ is \textit{tentacular} if it satisfies a number of carefully chosen growth conditions at infinity (see Definition \ref{def:TentHam}). For the reader not familiar with these conditions\footnote{The authors are aware the complement of this set of readers may consist solely of the authors.}, for the purposes of this Introduction they may regard ``tentacular'' to mean: either $A_1$ is symplectically diagonalisable or all eigenvalues $ \lambda_i$ of $\mathbb{J}A_1$ have $|\! \operatorname{Re} (\lambda_i)|>2$.  In this case we call $ \Sigma$ a \textit{tentacular hyperboloid}. 
\begin{theorem}\label{thm:compRFH}
Let $H$ be a tentacular quadratic Hamiltonian of the form \eqref{defH}. Then the Rabinowitz Floer homology of $H$ is given by:\footnote{For $k=n-1, RFH_{1-n}(H)=\Z_2\oplus \Z_2$.}
\[
RFH_*(H) \coloneqq \left\lbrace
\begin{array}{c c c}
\Z_2 &  &  *=1-n,-k,\\
0 & & \textrm{otherwise.}
\end{array}
\right.
\]
In particular $RFH_*(H)\neq 0$ and $RFH_*(H)\neq H_{*+n-1}(\Sigma)$.
\end{theorem}
The Weinstein Conjecture is obvious for the hyperboloids contained in Theorem \ref{thm:compRFH}. One of the strengths of Floer theoretical calculations, however, is that they are invariant under controlled perturbations. Therefore as an immediate corollary to Theorem \ref{thm:compRFH}, we obtain:
\begin{cor}
\label{cor:WC}
If $H$ is a Hamiltonian as in Theorem \ref{thm:compRFH} and $\{\Sigma_s\}_{s\in [0,1]}$ is a smooth $1$-parameter family of compact perturbations of $\Sigma\coloneqq H^{-1}(0)$ through strongly tentacular hypersurfaces, then each $\Sigma_s$ carries a closed characteristics.
\end{cor}

\begin{remark}
\label{k0n}
	\textbf{The cases $k = 0$ and $k = n$:} We assume throughout this paper that $1 \le k \le n-1$. Let us briefly comment on the  two other cases. For $k=0$, the Hamiltonian $H$ from \eqref{defH} is of the form $H=H_1-1$, and the proof of Lemma \ref{lem:orbits} below shows that the Hamiltonian vector field of $H_1$ on $ \Sigma$ has no closed orbits. Thus Theorem \ref{thm:compRFH} holds trivially as $RFH(H)$ agrees by definition with the Morse homology of $\Sigma$ (with shifted degree). Meanwhile if $k = n$ then $H = H_0$, and the hypersurface is compact. The Rabinowitz Floer homology for such hypersurfaces vanishes by \cite[Thm.\ 1.2]{CieliebakFrauenfelder2009}. The Weinstein Conjecture is trivially false for $k =0$, while it holds true for $k=n$, as proved by Viterbo in \cite{viterbo_1987} for compact contact type hypersurfaces in $T^* \mathbb{R}^n$.
\end{remark}

\begin{remark}
Symplectic homology has also recently been extended to the non-compact setting. In \cite{cieliebak_eliashberg_polterovich_2017}, Cieliebak, Eliashberg and Polterovich define and compute the symplectic homology for a certain subclass of the symplectic hyperboloids we consider here. Their computations are consistent with ours, and we conjecture that the long exact sequence relating symplectic homology and Rabinowitz Floer homology \cite{CieliebakFrauenfelderOancea2010} extends to the non-compact setting. In a slightly different direction, Ganatra, Pardon and Shende \cite{ganatra_pardon_shende_2019} have defined the symplectic homology for a class of Liouville manifolds with boundary -- Liouville sectors -- using the compactness methods of Groman \cite{groman_2020}. It seems likely that the hypersurfaces we consider here can also be fitted into this framework. We hope to discuss this elsewhere. 
\end{remark}

\textsc{Outlook:} In the compact world, both Rabinowitz Floer homology and symplectic homology have been profitably used to study orderability problems for compact contact manifolds \cite{FraserPolterovichRosen2012,AlbersFuchsMerry2015, albers_fuchs_merry_2017,albers_merry_2018,chantraine}. Cieliebak, Eliashberg and Polterovich initiated the study of orderability problems for non-compact contact manifolds in \cite{cieliebak_eliashberg_polterovich_2017} using symplectic homology. Our companion computation shows that Rabinowitz Floer homology is also well suited to this problem. We will return to these questions in a sequel to the present paper.

In another direction, we note that our results are consistent (as they should be!) with van den Berg, Pasquotto and Vandervorst's earlier Weinstein Conjecture \cite{BergPasq2009} results for non-compact hypersurfaces in $T^*\mathbb{R}^n$, which were based on variational methods. These results were later extended to cover unit cotangent bundles of Riemannian manifolds with flat ends \cite{berg_pasquotto_rot_vandervorst_2016,suhr_zehmisch_2016}. It is an interesting -- albeit, formidable -- problem to try and generalise the Floer-theoretical methods used in the present article to cover this setting. 

Finally, Miranda and Oms \cite{miranda_oms_2020} have very recently used the methods from $b$-symplectic and contact geometry to study the Weinstein Conjecture for certain non-compact hypersurfaces, including examples in the planar restricted circular three-body problem.\\

\textsc{Sketch of the proof of Theorem \ref{thm:compRFH}:}
The \textit{Rabinowitz action functional} for a Hamiltonian $H\colon  T^*\R^n\rightarrow\mathbb{R}$ associates to a pair $(v,\eta)$ of a loop $v\colon S^1 \to  T^*\R^n$ and a real number $\eta$ its action by
\[
\mathcal{A}^H(v, \eta) \coloneqq  \int_{S^1}v^{*}\lambda-\eta\int_{S^1}H(v(t))\,dt.
\]
Here, $\lambda$ is a primitive of $\omega = d p \wedge d q$. The critical set $\Crit(\cA^H)$ of $\cA^H$ consists of pairs $(v,\eta)$, such that $v(S^1)\subset H^{-1}(0)$ and $\partial_t v= \eta X_H(v)$, where $X_H$ is the Hamiltonian vector field of $H$. The positive $L^2$ gradient equation for this functional,  which we call the Rabinowitz Floer equations for $H$, is the following Floer equation for $v\colon \R \times S^1 \rightarrow T^*\mathbb{R}^n$ coupled with an ODE for $\eta\colon  \R \rightarrow \R$:
\begin{equation*}
 \left(\begin{array}{c}
\partial_{s}v\\
\partial_s \eta
\end{array}\right)
 = \left( \begin{array}{c}
-J(v,\eta,t)\bigl[\partial_{t}v-\eta X_{H}(v)\bigr] \\
-{\int_0^1 H(v) \,dt \quad}\end{array} \right).
\end{equation*}
The standard counting of rigid solutions of this equation defines a boundary operator
\[
\partial \colon  CF_*(H,f) \rightarrow CF_{*-1}(H,f),
\]
on a graded $\Z_2$-vector space $ CF_*(H,f)$ generated by critical points of an auxiliary  coercive Morse function $f$ on $\Crit(\cA^H)$, which generically is a countable union of finite dimensional manifolds. The $\Z$-grading of $CF_*(H,f)$ is defined by the transverse Conley-Zehnder index of periodic orbits plus the signature index of critical points of $f$. We use the $ \mathbb{Z}$-grading convention of \cite{CieliebakFrauenfelderOancea2010}, which differs from the $ \mathbb{Z}+ \frac{1}{2}$ grading convention of \cite{CieliebakFrauenfelder2009} by a factor of $1/2$. The homology of the chain complex $(CF_*(H,f),\partial)$ is called the \textit{Rabinowitz Floer homology} of $H$ and is denoted by $ RFH_*(H)$. For the Hamiltonians $H$ that we consider in Theorem \ref{thm:compRFH}, the fact that $ RFH_*(H)$ is well defined and independent of the auxiliary data used to construct it is proved in \cite{pasquotto2017, pasquotto2018}.

One can also play the same game with the Hamiltonian $H_0$ on $T^*\mathbb{R}^k$, thus yielding another Rabinowitz Floer homology $ RFH_*(H_0)$. This construction is rather easier as $ \Sigma_0$ is compact (and falls under the remit of the setup originally conceived in \cite{cieliebak2009}). However as $ \Sigma_0$ is displaceable in $ T^*\mathbb{R}^k$, by \cite[Thm.\ 1.3]{cieliebak2009} the Rabinowitz Floer homology of $H_0$ is not particularly interesting 
\[
RFH_*(H_0) = 0.
\]
All is not lost though: if one restricts to the subcomplex generated by orbits $(v,\eta)$ with $ \eta > 0$ -- a gadget we refer to as the \textit{positive Rabinowitz Floer homology} and denote by $RFH_*^+(H_0)$ -- then we obtain something non-zero:
\[
RFH^+_*(H_0) = \begin{cases} \mathbb{Z}_2, & * = k,\\
0, & \text{otherwise.}
\end{cases}
\]
The proof of Theorem \ref{thm:compRFH} uses the ``hybrid problem'' technique pioneered by Abbondandolo and Schwarz in their seminal paper \cite{AbbondandoloSchwarz2010}. The starting observation is that any periodic orbit of $X_H$ on $ \Sigma$ takes the form 
\[
\gamma(t) = ( \gamma_0(t), 0),
\]
where $ \gamma_0$ is a periodic orbit of $X_{H_0}$ on $ \Sigma_0$. This sets up an inclusion
\[
i \colon  \Crit(\cA^{H_0}) \hookrightarrow \Crit(\cA^H),
\]
which restricted to the non-constant orbits gives a bijection between $\Crit(\cA^{H_0})\setminus (\Sigma_0{\times}\{0\})$ and $\Crit(\cA^H)\setminus (\Sigma{\times}\{0\})$.
Despite this relationship between their critical points, there is no obvious relation between the space of negative gradient flow lines for these two functionals, and hence no reason to hope that $i$ induces a chain map. Nevertheless, we prove that for a particular choice of Morse functions $f$ and $f_0$:
\begin{theorem}
\label{thm:iso}
There exists a chain map 
\[
\psi \colon CF_*(H,f) \to CF_*(H_0,f_0),
\]
which induces a sequence of homomorphisms $\Psi$, such that the following diagram of two long exact sequences commutes:
\[
\xymatrix{
 \dots \ar[r] &H_{*+n -1}(\Sigma) \ar[r] \ar[d]_{\Psi^0}
                 & RFH_{*}^{\geq 0}(H) \ar[r] \ar[d] &    RFH^+_*(H) \ar[r] \ar[d]^{\Psi^+} & \dots \\
\dots \ar[r]  & H_{*+k -1}(\Sigma_0)  \ar[r]   & RFH_{*}^{\geq 0}(H_0) \ar[r] & RFH^+_*(H_0) \ar[r]  & \dots \\
             }
\]
Moreover $\Psi^+$ is an isomorphism, whereas $\Psi^0$ is a composition of an isomorphism coming from the retraction  $\Sigma\cong S^{n+k-1}\times\mathbb{R}^{n-k}\rightarrow S^{n+k-1}$ and an Umkehr map.
\end{theorem}
The construction of $\psi$ is based on counting solutions of the following hybrid problem: we consider tuples $(v^-,\eta^-,v^+,\eta^+)$ where
\begin{align*}
v^+ \colon  [0,\infty) \times S^1 &\rightarrow T^*\mathbb{R}^n,&  \eta^+ \colon  [0, \infty) &\to \R,\\
v^- \colon  (-\infty,0] \times S^1 &\rightarrow  T^*\mathbb{R}^k,& \eta^- \colon  (- \infty, 0] &\to \R,
\end{align*}
are solutions of the Rabinowitz Floer equations, $(v^+,\eta^+)$ for $H$ and $(v^-,\eta^-)$ for $H_0$, with prescribed asymptotics at $\pm\infty$ and satisfying the following coupling condition at $s=0$: identifying $T^*\mathbb{R}^n \cong T^*\mathbb{R}^k \times  \mathbb{R}^{n-k} \times \mathbb{R}^{n-k}$, we require
\[
 v^+(0,t) = (v^-(0,t), *, 0), \qquad \eta^-(0) = \eta^+(0).
\]
This coupling condition can be seen as a Lagrangian boundary condition for $\mathbb{R}^{2n+2k}$-valued maps on a half-cylinder, from which it follows that the hybrid problem is Fredholm. Precompactness of the moduli spaces of solutions $(v^-,\eta^-,v^+,\eta^+)$ with fixed asymptotes is established in Section \ref{ssec:bounds}. The chain map $ \psi$ is then defined by counting rigid solutions of the hybrid problem. To prove the remaining statements of Theorem \ref{thm:iso}, we show that automatic transversality holds at stationary solutions of the hybrid problem. Combining this with sharp energy estimates implies that if we order the critical points of $\mathcal{A}^H$ and $ \mathcal{A}^{H_0}$ by increasing action, then the matrix representation of $ \psi$ is upper triangular, and moreover the diagonal entries are all equal to 1, except for a single 0 coming from the minimum of $f$ on $ \Sigma$. From this, Theorem \ref{thm:iso} follows, and hence so does Theorem \ref{thm:compRFH}. \\

\textit{Acknowledgements:} The first and third authors are supported by the SNF grant \href{http://p3.snf.ch/project-182564}{Periodic orbits on non-compact hypersurfaces}. We thank Kai Cieliebak for pointing out \cite{cieliebak_eliashberg_polterovich_2017} to us, and for helpful remarks during the preparation of this article.
\section{Preliminaries}
\label{section_setting}
We begin with a brief discussion of Rabinowitz Floer homology, with a special emphasis on the non-compact framework. To keep the exposition concise we place ourselves throughout in the linear setting of $T^*\mathbb{R}^m$. Whilst this restriction is at present necessary for the construction of Rabinowitz Floer homology for non-compact hypersurfaces, this is by no means the case for compact hypersurfaces. We refer the reader to the survey article \cite{albers2012} for a leisurely introduction to the various settings that (compact) Rabinowitz Floer homology can be defined.\\

\textbf{Sign Conventions}: Let $\omega = d p \wedge dq$ denote the standard symplectic form on $T^* \mathbb{R}^m = \mathbb{R}^m \times (\mathbb{R}^m)^*$. We identify $T^* \mathbb{R}^m$ with the complex vector space $(\mathbb{C}^m, i)$ via the map $(q,p) \mapsto q + ip$, and denote by $\mathbb{J}$ the corresponding complex structure.\footnote{As a general notational guide to the reader, throughout the rest of this article, we work with compact hypersurfaces in $T^*\mathbb{R}^k$ and non-compact hypersurfaces in $T^* \mathbb{R}^n$, where $ 0< k < n$. In this preliminary section we treat both cases simultaneously, and hence use the letter $m$ instead.} Explicitly, this means that 
\begin{equation}
    \label{pref-J}
   \mathbb{J} = \begin{pmatrix} 0 &\operatorname{Id} \\  -\operatorname{Id} & 0 \end{pmatrix}.
\end{equation}
Let $g_{\mathbb{J}} \coloneqq  \omega( \cdot , \mathbb{J} \cdot )$, so that $g_{\mathbb{J}}$ is the real part of the standard Hermitian structure on $\mathbb{C}^m$, and hence a Riemannian metric on $ T^*\mathbb{R}^m$. Sometimes it will be necessary to include the dimension in our notation, in which case we write $ \omega_m$, $\mathbb{J}_m$ and so on. We use the sign convention that the symplectic gradient/Hamiltonian vector field $X_H$ of a Hamiltonian $H \colon T^*\mathbb{R}^m \to \mathbb{R}$ is given by $\omega(X_H, \cdot ) = -dH$, so that the Poisson bracket of two Hamiltonians on $T^*\mathbb{R}^m$ is given by $\{ F,H\} \coloneqq  \omega(X_F,X_H)$.\\

\textbf{Hamiltonians:} We now introduce the class of Hamiltonians that we work with. A vector field $Y$ on $ T^* \mathbb{R}^m$ is said to be a \textit{Liouville vector field} if $ d( \omega(Y, \cdot )) = \omega$. A Liouville vector field $Y$ is said to be \textit{asymptotically regular} if  $\Vert DY(x)\Vert \le c$ for some positive constant $c$ and all $x\in T^*\R^m$.
\begin{definition}\label{def:TentHam}
Let $\mathcal{H}^*$ denote the set of Hamiltonians $H$ on $T^*\mathbb{R}^m$ such that \textit{either} 
\begin{enumerate}
\item[(c)\namedlabel{c1}{(c)}] $dH$ is compactly supported,
\end{enumerate}
or the following three axioms are satisfied:
\begin{enumerate}
\item[(h1)\namedlabel{h1}{(h1)}] there exists an asymptotically regular Liouville vector field $Z$ and constants $c,c'>0$,
such that $dH(Z)(x) \ge c|x|^2 - c'$, for all $x\in T^*\R^m$;
\item[(h2)\namedlabel{h2}{(h2)}] (sub-quadratic growth) $\sup_{x\in T^*\R^m} \Vert D^3 H(x)\Vert \cdot |x| <\infty$;
\item[(h3)\namedlabel{h3}{(h3)}] in the neighbourhood of $H^{-1}(0)$ exists a coercive function $F$, such that for all $x\in H^{-1}(0)$ either $\{H,F\}(x)\neq 0$ or $\{H,\{H,F\}\}(x)>0,$ as $|x|\to\infty$.
\end{enumerate}
\end{definition}
\begin{remark}
Note that if $H \in \mathcal{H}^*$ and $h \in C^{ \infty }_c(T^*\mathbb{R}^n)$ then also $H+h\in\mathcal{H}^*$. 
\end{remark}
We say that $H^{-1}(0)$ is of \textit{contact type} if there exists an asymptotically regular Liouville vector field $Y$ such that $dH(Y)(x) >0 $ for all $x \in H^{-1}(0)$. Note that this implies that $H^{-1}(0)$ is a smooth hypersurface.
\begin{definition}
\label{defn-tenta}
Let $ \mathcal{H} \subseteq \mathcal{H}^*$ denote the subset of those Hamiltonians $H$ which in addition have the property that $H^{-1}(0)$ is of contact type.
We say that Hamiltonians $H \in \mathcal{H}$ satisfying \ref{h1}, \ref{h2} and \ref{h3} are \textit{strongly tentacular}. If we drop the requirement that the function $F$ in \ref{h3} is coercive then $H$ is called simply \textit{tentacular}. 
\end{definition}

\begin{remark}
The strongly tentacular condition implies that all the periodic orbits of $H$ are contained in a compact set. This may not be the case for tentacular Hamiltonians. Rabinowitz Floer homology is defined for Morse-Bott tentacular Hamiltonians whose orbits are contained in a compact set. Invariance under compact perturbation requires the strongly tentacular condition. This explains why the adjective ``strongly'' appears in Corollary \ref{cor:WC} but not in Theorem \ref{thm:compRFH}.
\end{remark}

The connection between the definition of strongly tentacular Hamiltonians given here and the one presented in the introduction is given by the following result:

\begin{prop}\label{prop:TentEx}
Let $H$ be a Hamiltonian of the form \eqref{defH}. Then $H$ belongs to $ \mathcal{H}$, if the Jordan decomposition of $\mathbb{J}A_1$ has $m_i\times m_i$ blocks corresponding to eigenvalues $\lambda_i$, where each pair $(m_i,\lambda_i)$ satisfies one of the following conditions:
\begin{enumerate}[label=\roman*)]
\item
\label{case1}
$m_i=1$ and $\operatorname{Re}(\lambda_i)\neq 0$;
\item $m_i=2$ and $|\!\operatorname{Re}(\lambda_i)|>\frac{1}{\sqrt{2}}$;
\item 
\label{case3}
$m_i>2$ and $|\!\operatorname{Re}(\lambda_i)|>2$.
\end{enumerate}
\end{prop}
\begin{remark}
\label{works-for-any-tent}
Proposition \ref{prop:TentEx} provides a concrete class of examples for which our main results, Theorem \ref{thm:compRFH} and Theorem \ref{thm:iso} are valid. Note that cases \ref{case1} and \ref{case3} correspond to our ``definition'' of strongly tentacular on page \pageref{thm:compRFH}. We emphasise however that Theorems \ref{thm:compRFH} and \ref{thm:iso} are valid for any tentacular Hamiltonian $H$ of the form \eqref{defH} with $A_0$ positive definite and $\mathbb{J}A_1$ hyperbolic\footnote{That is, $\mathbb{J}A_1$ has no purely imaginary eigenvalues.}.
\end{remark}

Below we will outline the construction of Rabinowitz Floer homology groups for Hamiltonians $H \in \mathcal{H}$. In the case $H$ satisfies \ref{c1}, this reduces to the original definition presented by Cieliebak and Frauenfelder \cite{CieliebakFrauenfelder2009}, only specialised to $T^*\mathbb{R}^m$. Meanwhile for strongly tentacular $H$, the construction\footnote{ Actually the construction for strongly tentacular Hamiltonians subsumes the compact contact type hypersurfaces as a special case, see Remark \ref{compact-special-case} below.} comes from \cite{pasquotto2018}.\\

\textbf{Complex structures:} Although $T^*\mathbb{R}^m$ comes equipped with a preferred choice $\mathbb{J}$ of complex structure, in order to achieve transversality for the various moduli spaces used in the definition of Rabinowitz Floer homology, we are forced to work with generic data. To this end we now introduce a suitable parameter space of almost complex structures. 

Let $ \mathcal{J}$ denote the space of all \textit{compatible} almost complex structures $J$ on $T^* \mathbb{R}^m$. Here compatible means that $g_J \coloneqq \omega( \cdot ,J \cdot  )$ is a Riemannian metric on $T^*\mathbb{R}^m$. We view $ \mathcal{J}$ as a pointed space with basepoint $ \mathbb{J}$. An easy linear algebra argument (see for example  \cite[Prop.\ 13.1]{Silva2001}) shows that $ \mathcal{J}$ is contractible. 

Fix an open set $V \subset T^*\mathbb{R}^m\times \mathbb{R}$, and let $ \mathcal{J}(V,\mathbb{J})$ denote the set of smooth maps
\[
(t,\eta) \mapsto J(\cdot,\eta, t) \in \mathcal{J}, \qquad (t, \eta) \in S^1 \times \mathbb{R},
\]
such that 
\[
J( x , \eta, t) = \mathbb{J}, \qquad \text{whenever}\qquad (x, t) \notin V,
\]
and such that
\begin{equation}
\sup_{(t,\eta)\in S^1 \times \R}\|J(\cdot,\eta,t)\|_{C^l}<+\infty, \qquad \forall\ k\in \N.
\label{supJ}
\end{equation}
Here the $C^l$ norm is taken with respect to the standard Riemannian metric $g_{ \mathbb{J}}$ on $T^*\mathbb{R}^m$. Finally we let $\mathcal{J}_{\star}$ denote the union of all spaces $\mathcal{J}(V, \mathbb{J})$ for $V$ of the form $ \operatorname{int} K  \times \mathbb{R} {\setminus} [-a,a]$ for $K \subseteq T^*\mathbb{R}^m$ compact and $a >0$, equipped with the colimit topology. We view $ \mathcal{J}_\star$ as another pointed space, with basepoint $ \mathbb{J}$. This space is again easily seen to be contractible. We write $\mathcal{J}_\star^m$, if we have to indicate the dimension of the underlying space $T^*\R^m$.\\

\textbf{The Rabinowitz action functional:} 
The free period action functional---or \textit{Rabinowitz action functional}---is defined by
\begin{align*}
    \mathcal{A}^{H} & \colon C^{\infty}(S^1, T^*\mathbb{R}^m)\times\mathbb{R}\to\mathbb{R} \\
 (v, \eta) & \mapsto \int_{S^1}v^{*}\lambda-\eta\int_{S^1}H(v(t))\,dt.
\end{align*}
Here $ \lambda$ is any primitive of $ \omega$, for example $\lambda=\frac{1}{2}(pdq-qdp)$. The real number $ \eta$ can be thought of as a Lagrange multiplier version of the area functional from classical mechanics. Thus critical points $ \mathcal{A}^H$ are critical points of the area functional restricted to the space of loops with $H$ mean value zero. Moreover, since $H$ is invariant under its own Hamiltonian flow, the mean value constraint can be upgraded to a pointwise constaint, and we obtain:
\begin{lem}
A pair $(v, \eta)$ is a critical point of $ \mathcal{A}^H$ if and only if\footnote{For $ \eta = 0$ this should be interpreted as: $t \mapsto v(t)$ is constant.} $t \mapsto v(t/\eta)$ is a closed orbit of $X_H$ contained in $H^{-1}(0)$.
\end{lem}~

\textbf{The Morse-Bott condition:}
It is not just the complex structure that needs to be chosen generically. Floer theory also requires us to work with a Hamiltonian which satisfies a certain generic non-degeneracy condition. Unlike the case of complex structures however, the non-degeneracy condition admits a direct definition. Throughout the rest of this section, we denote by $ \Sigma$ the level set $H^{-1}(0)$ of a given Hamiltonian $H \in \mathcal{H}$ and we denote by $Y$ an asymptotically regular Liouville vector field such that $dH(Y)|_\Sigma > 0$.

\begin{definition} 
$ \ $ 
\begin{enumerate}[label=(\roman*)]
    \item We say that the Rabinowitz action functional $\cA^H$ is \emph{Morse-Bott} if the critical set of $\cA^H$ is a discrete union of finite dimensional manifolds and for every connected component $\Lambda\subseteq \Crit(\cA^H)$ and every $x\in \Lambda$
\begin{equation}\label{eq:MorseBott}
T_x\Lambda=\ker(\nabla^2_x\cA^H).
\end{equation}
\item We say that the closed orbits of the flow $\phi^t$ of $X_H$ on $\Sigma$ are of \emph{Morse-Bott type} if $\eta$ is constant on every connected component $\Lambda \subseteq \Crit({\cA^{H}})$, and the image $\mathcal{N}^\eta$ of a connected component $\Lambda$ with period $\eta$ under the projection $(v,\eta) \mapsto v(0)$ is a closed submanifold of $\Sigma$, such that for all $x\in \mathcal{N}^\eta$ we have
\begin{equation}\label{def:MBT}
T_{p}\mathcal{N}^\eta = \ker (D_{p}\phi^{\eta} - \operatorname{Id}) \cap T_{p} \Sigma.
\end{equation}
\end{enumerate}
\end{definition}

\begin{remark} 
The two Morse-Bott conditions are closely related. If $\cA^{H}$ is a Morse-Bott functional  and $\eta$ is constant on every connected component of $\Crit({\cA^{H}})$, then by \cite[Lem.\ 3.3]{pasquotto2018} all closed orbits of the Hamiltonian flow $\phi^t$ on $\Sigma$ are of Morse-Bott type. Conversely, if $H$ is \textit{defining} for $\Sigma$, i.e.\ if $dH(Y)|_\Sigma\equiv 1$, and if all periodic orbits are of Morse-Bott type then by \cite[Lem.\ 20]{Fauck2016} the corresponding Rabinowitz action functional $\cA^H$ is Morse-Bott.
\end{remark}

The Morse-Bott property of the Rabinowitz action functional is typically achieved by perturbing slightly the Hamiltonian function. However in our case we will calculate the Rabinowitz Floer homology of the specific Hamiltonians satisfying \eqref{defH} by hand (so to speak), and therefore in Section \ref{sec:Morse-Bott} we will check directly that these Hamiltonians fulfill the Morse-Bott property.\\

\textbf{The Rabinowitz Floer equation:}
Fix $ J \in \mathcal{J}_\star$. The positive $L^2$ gradient equation $ \partial_s u = \nabla_J \mathcal{A}^H$ for $u = (v, \eta)$ is the following Floer equation for $v$ coupled with an ordinary differential equation for $\eta$:
\begin{equation}
\label{eq-grad-mas2}
 \left(\begin{array}{c}
\partial_{s}v\\
\partial_s \eta
\end{array}\right)
 = \left( \begin{array}{c}
-J(v,\eta,t)\bigl[\partial_{t}v-\eta X_{H}(v)\bigr] \\
-{\int_0^1 H(v)\,dt \quad}\end{array} \right).
\end{equation}
A solution $u$ of \eqref{eq-grad-mas2} with finite $L^2$ energy $\int_\R \|\partial_s u\|^2ds<\infty$, is called a \emph{Floer trajectory}. Given two distinct components $\Lambda^-$ and $\Lambda^+$ of $\operatorname{Crit} (\mathcal{A}^H)$, we denote by $\mathcal{M}(\Lambda^-,\Lambda^+)$ the set of all solutions of \eqref{eq-grad-mas2} with finite energy and ${\lim_{s\to\pm \infty}u(s)\in \Lambda^\pm}$. We define $\operatorname{ev}^\pm \colon  \mathcal{M}(\Lambda^-,\Lambda^+)\to\Lambda^\pm$ to be the evaluation maps:
\begin{equation}
\operatorname{ev}^-(u)\coloneqq \lim_{s\rightarrow -\infty} u(s,t)\quad \textrm{and}\quad \operatorname{ev}^+(u)\coloneqq \lim_{s\rightarrow +\infty} u(s,t). \label{eqn:ev}
\end{equation}

\begin{definition}
\label{def:regular}
We say that a couple $(H,J)\in \mathcal{H} \times \mathcal{J}_\star$ is \textit{regular} if it satisfies the following two conditions:
\begin{enumerate}[label=\roman*)]
\item The Rabinowitz action functional $ \mathcal{A}^H$ is Morse-Bott, and the closed orbits of $X_H$ on $\Sigma$ are of Morse-Bott type;
\item For every pair of connected components $\Lambda^{-},\Lambda^{+}\subseteq \Crit({\cA^H})$ the associated moduli space $\mathcal{M}(\Lambda^-, \Lambda^+)$ is a smooth finite dimensional manifold without boundary.
\end{enumerate}
\end{definition}
By \cite[Lem.\ 8.7]{pasquotto2018} the set of such regular couples is comeagre in $\mathcal{H} \times \mathcal{J}_\star$. From now on we assume that $(H,J)$ is regular. Next, we introduce flow lines with cascades following \cite{Frauenfelder2004}: 
Consider a Morse function $f \colon \operatorname{Crit} (\mathcal{A}^H ) \to \mathbb{R}$ such that $f$ restricts to a coercive\footnote{When $H$ satisfies \ref{c1} this is condition is automatic.} function on $\Sigma$. Fix a Riemannian metric on $\operatorname{Crit} (\mathcal{A}^H)$ having a Morse-Smale gradient flow $\phi^t$. 
For $z \in \operatorname{Crit} (f)$, we denote by $W^s_f(z)$ and $W^u_f(z)$ the (un)stable manifolds with respect to $\phi^t$. A flow line with $k \geq 1$ cascades between $z^-,z^+\in \Crit(f)$ belonging to distinct connected components is a tuple $ (u_{1},...,u_{k}, t_1, ..., t_{k-1})$,
where each $u_i$ is a non-stationary Floer trajectory (cf.\ \eqref{eq-grad-mas2}), such that
\begin{gather*}
\phi^{t_i}\circ \operatorname{ev}^+(u_i)=\operatorname{ev}^-(u_{i+1}),\qquad i=1, ..., k-1, \\
 \operatorname{ev}^-(u_1)\in W^u_f(z^-), \qquad \operatorname{ev}^+(u_k)\in W^s_f(z^+).
\end{gather*}
We denote the set of all flowlines with $k$ cascades from $z^-$ to $z^+$ as $\mathcal{M}_{\operatorname{cas}}^k(z^-,z^+)$. There is a natural $\R^k$ action on $\mathcal{M}_{\operatorname{cas}}^k(z^-,z^+)$ given by $u_i(\cdot) \mapsto u_i(a+\cdot)$ and we define the space of all flow lines with cascades from $z^-$ to $z^+$ by
\[
\mathcal{M}(z^-,z^+)\coloneqq\bigcup_{k \geq 1}\left(\mathcal{M}_{\operatorname{cas}}^k(z^-,z^+)\mspace{-5mu}\left/ \R^k\right.\right).
\]
By taking such a union, we obtain another smooth manifold without boundary, which moreover is compact ``up to breaking''. 

If $z^-$ and $z^+$ are critical points of $f$ belonging to the same component of $ \operatorname{Crit} (\mathcal{A}^H)$ we set $\mathcal{M}(z^-,z^+)$ to be the quotient $W^u(z^-) {\cap} W^s(z^+)\big/\mathbb{R}$ arising from the natural $ \mathbb{R}$ action by translation.\\

\textbf{Grading:}
We use the $ \mathbb{Z}$-grading convention of \cite{CieliebakFrauenfelderOancea2010}, which differs from the $ \mathbb{Z}+ \frac{1}{2}$ grading convention of \cite{CieliebakFrauenfelder2009} by a factor of 1/2. Explicitly, for $z \in \operatorname{Crit} (f)$ we define
\begin{equation}\label{muRF}
\mu(z) \coloneqq \tCZ(z) + \mu_\sigma(z)  +\frac{1}{2},
\end{equation}
where $\tCZ$ is the transverse Conley-Zehnder index and $\mu_\sigma$ is the signature index defined by
\begin{equation}\label{musig}
\mu_\sigma(x)=\frac{1}{2}\big(\dim W^s_f(x)-\dim W^u_f(x) \big).
\end{equation}
For a regular pair $(H,J)$, one has 
\[
\dim \mathcal{M}(z^-,z^+) = \mu(z^+)- \mu(z^-) - 1.
\]
The compactness up to breaking property alluded to above tells us that when $\mu(z^+)- \mu(z^-) = 1$ the space $ \mathcal{M}(z^-,z^+)$ is compact, and hence a finite set. We denote by $n(z^-,z^+)$ its parity.\\

\textbf{The Rabinowitz Floer complex:}
We define the $\Z_2$-vector space $CF(H,f)$ as the set of formal sums  of the form $ \sum_{z \in S} z$, where $S \subset \operatorname{Crit}(f)$ is a (possibly infinite) set satisfying the \textit{Novikov finiteness condition}
\begin{equation}
\#\;\big\{ z\in S\,\big|\,\mathcal{A}^H(z)>a\big\} < \infty\qquad \forall \, a\in\mathbb{R}.
\label{novikov}
\end{equation}
We denote by $CF_k(H,f) \subset CF(H,f)$ those sums $\sum_Sz$ with $ \mu(z) = k$ for all $z \in S$. We turn $CF_*(H,f)$ into a chain complex by defining 
\[
\partial z^+ \coloneqq \sum n(z^-,z^+) z^-, 
\]
where the sum is taken over all critical points $z^-$ with $ \mu(z^+) =\mu(z^-)+1$, and then extending by linearity. Compactness up to breaking implies that $ \partial^2 = 0$, and a continuation argument tells us that the resulting \textit{Rabinowitz Floer homology} $RFH(H)$ is independent of the auxiliary data used to define it. We refer the reader to \cite{CieliebakFrauenfelder2009} (when $H$ satisfies \ref{c1}) and \cite{pasquotto2018} (when $H$ is strongly tentacular) for details. 
\begin{remark}
\label{rem:onlyonsigma}
If $H \in \mathcal{H}$ satisfies condition \ref{c1} then the Rabinowitz Floer homology groups only depend on $H$ through its zero level set $\Sigma$, and thus we could write $RFH_*(\Sigma)$ instead (although we won't). For strongly tentacular $H$ this need not be the case. However for $H$ of the form \eqref{defH} the main result of this article, Theorem \ref{thm:compRFH}, shows that the Rabinowitz homology groups only depend on the ``compact part'' of the zero level set. 
\end{remark}
\begin{remark} \label{compact-special-case}
In fact, when $ \Sigma$ is compact, there is considerable freedom in the choice of the Hamiltonian $H$ realising $\Sigma$ as its zero level set. In this section for historical reasons we have concentrated on the case where $dH$ is compactly supported, but one could also use a quadratic Hamiltonian \cite{AbbondandoloSchwarz2010}. In particular, for the hypersurface $\Sigma_0$ from \eqref{Sigma-Sigma0} we can use the Hamiltonian $H_0$ from \eqref{defH} to compute its Rabinowitz Floer homology. This fact will be used in the proof of Theorem \ref{thm:iso}. This also shows that the compact case is subsumed by the  strongly tentacular case, i.e.\ if $H \in \mathcal{H}$ satisfies condition \ref{c1} then there exists a strongly  tentacular $H' \in \mathcal{H}$ such that $H^{-1}(0) = (H')^{-1}(0)$.  
\end{remark}

\textbf{Positive Rabinowitz Floer homology:} The action functional $\cA^H$ provides an $\R-$filtration on $CF\left(\cA^H,f\right)$ as follows: For $t\in \R$ denote the complex generated by critical points of action $\leq t$ by
\begin{equation}
CF^{\leq t}(H,f)\coloneqq \left\lbrace {\textstyle \sum_{z\in S}z}\in CF(H,f) \left|\ \sup_{z\in S}\mathcal{A}^H(z) \leq t\right.\right\rbrace.
\label{CFt}
\end{equation}
The boundary operator does not increase the action, i.e.\
\begin{equation}\label{filtration}
\partial\left( CF^{\leq t}_{*+1}(H,f)\right)\subseteq CF^{\leq t}_*(H,f).
\end{equation}
The positive Rabinowitz Floer homology $RFH^+(H)$ is the homology of the following quotient complex generated by the critical points with positive action: 
\begin{equation}\label{CF+}
\begin{aligned}CF^+(H,f) & \coloneqq CF(H,f)\Big/CF^{\leq 0}(H,f)\\
& \, =\left\lbrace x\in \Crit(f)\ \big|\ \cA^H(x)>0\right\rbrace\otimes \Z_2.\end{aligned}
\end{equation}
The associated boundary operator $\partial^{\scriptscriptstyle +}$ is induced by $\partial$ on the quotient, which is well-defined as $\partial$ reduces the action (cf.\ (\ref{filtration})). Geometrically, $\partial^+$ is defined by counting flow lines with cascades where both endpoints have strictly positive action. Analogously, one defines $RFH^-(H)$ as the homology of $CF^{<0}(H,f)$.

More generally, $CF^+(H,f)$ fits into the following short exact sequence of complexes induced by action filtration:
\begin{equation}\label{ActionShortExSeq}
0\rightarrow \equalto{CF^0(H,f)}{CF^{\leq 0}/CF^{<0}}\rightarrow \equalto{CF^{\geq 0}(H,f)}{CF/CF^{<0}}\rightarrow \equalto{CF^+(H,f)}{CF/CF^{\leq 0}}\rightarrow 0,
\end{equation}
where the boundary operator of each complex is induced by $\partial$ and well-defined due to (\ref{filtration}). We hence obtain the following long exact sequence in homology
\begin{equation}\label{ActionLongExSeq}
\dots \rightarrow RFH^0_\ast(H)\rightarrow RFH^{\geq 0}_\ast(H)\rightarrow RFH^+_\ast(H)\rightarrow RFH^0_{\ast-1}(H)\rightarrow \dots
\end{equation}
where $RFH^0_\ast(H)\cong H_{\ast+n-1}(\Sigma)$, as $\partial$ counts in the zero-action window only flow lines with no cascades, i.e.\ Morse flow lines on $\Sigma$.

 The positive Rabinowitz Floer homology is independent of the auxiliary choices used to define it and invariant under compact perturbations. When $H$ satisfies condition \ref{c1}, these facts follow from \cite[Cor.\ 3.8]{CieliebakFrauenfelder2009}. The proof for strongly tentacular $H$ goes along similar lines, but the argument is slightly more involved, and has not yet appeared in the literature. Therefore in the next section we supply the full details. Similar statements apply to the other variants $RFH^{\ge 0}$, $RFH^-$, etc.
\section{Invariance of the positive Rabinowitz Floer homology}
The goal of this section is to show the following invariance of $RFH^+(H)$:
\begin{theorem} \label{thm:posRFHtent}
For a strongly tentacular Hamiltonian $H$ the positive Rabinowitz Floer homology $RFH^+(H)$ does not depend on the almost complex structure $J$ or on the Morse-Smale pair $(f,g)$ on $\operatorname{Crit} (\mathcal{A}^H )$. Moreover if $\{H_s\}$ is a $1$-parameter family of Hamiltonians in the affine space of compactly supported perturbations of $H$, then $RFH^+(H_s)$ is constant along $\{H_s\}$.
\end{theorem}
As seen above, the positive Rabinowitz Floer homology $RFH^+(H)$ is generated by non-constant periodic orbits and thus vanishes in the absence of closed characteristics. 

Throughout this section we will denote by $(M, \omega)$ any exact symplectic manifold. In order to prove Theorem \ref{thm:posRFHtent}, we need compactness results for homotopies of Hamiltonians and almost complex structures, which are stronger than the ones proved in \cite{pasquotto2017}. To obtain these results, we recall the notion of \emph{uniform continuity of \emph{(PO)}}, as introduced in \cite{pasquotto2018}:
\begin{definition}
We say that $H$ satisfies property (PO) if 
for any fixed action window, all non-degenerate periodic orbits are contained in a compact subset of $M$.
Moreover, we say that property (PO) is \emph{uniformly continuous} at $H$ if there exists an open neighbourhood $\mathcal{O}(H)$ of $0$ in $C^{\infty}_{c}(M)$ and an exhaustion of $M$ by compact sets $\{K_{n}\}_{n\in\mathbb{N}}$, such that for every $n\in \mathbb{N}$ and every $h\in \mathcal{O}(K_{n})=\mathcal{O}(H)\cap C^{\infty}_{0}(K_{n})$, whenever
\[
(v,\eta)\in \Crit(\cA^{H+h})\quad\textrm{and}\quad 0<\Big|\cA^{H+h}(v,\eta)\Big|\leq n,
\]
then $v(S^{1})\subseteq K_{n} \cap (H{+}h)^{-1}(0)$.
\end{definition}
\begin{remark}
By \cite[Lem.\ 8.4]{pasquotto2018} every strongly tentacular Hamiltonian satisfies the axiom of uniform continuity of (PO). 
\end{remark}
Define: $\displaystyle\qquad \Crit^+(\cA^H) \coloneqq \Crit(\cA^H )\cap \left(\cA^H\right)^{-1}((0,+\infty)).$
\begin{remark}\label{rem:C2small}
Let $H \colon T^*\R^n\to \R$ be a Hamiltonian on $(T^*\R^n,\omega_n)$. If $(v,\eta)\in \Crit^+(\cA^H)$, then the following adaptation of \cite[Lem.\ 2.2]{laudenbach2004} to the Rabinowitz Floer framework holds:
\[
|\eta| \sup_{v(S^1)}\|DX_H\|\geq 2\pi.
\]
\end{remark}
\begin{lem}\label{lem:delta1}
Let $H$ be a Hamiltonian on an exact symplectic manifold $(M,\omega)$, such that $\Sigma \coloneqq H^{-1}(0)$ is of contact type.
If property (PO) is uniformly continuous at $H$ then there exists an open neighbourhood $\mathcal{O}(H)$ of $0$ in $C^\infty_c(M)$, such that for every compact set $K\subseteq M$ there exists a constant $\delta>0$, such that for all $h\in \mathcal{O}(H)\cap C_c^\infty(K)$
\[
\inf \left\lbrace \cA^{H+h}(v,\eta)\ \Big| \ (v,\eta)\in \Crit^+(\cA^{H+h})\right\rbrace>\delta.
\]
\end{lem}
\begin{proof}
We argue by contradiction. Let $\widetilde{\mathcal{O}}(H)$ be the open neighbourhood from property (PO). Without loss of generality, we can assume that $K=K_n$ for some $K_n$ in the exhaustion of $M$. Suppose that there exists a sequence $h_n\in \widetilde{\mathcal{O}}(H)\cap C_c^\infty(K)$ and a sequence $(v_n,\eta_n)\in \Crit^+(\cA^{H+h_n})$, such that $\lim_{n\to\infty}h_n= 0$ and
\begin{equation}\label{A->0}
\lim_{n\to\infty}\cA^{H+h_n}(v_n,\eta_n)= 0. 
\end{equation}
Property (PO) is uniformly continuous at $H$, hence all the periodic orbits $v_n$ are contained in the compact set $K$. By requiring certain bounds on the derivatives for $h\in\widetilde{\mathcal{O}}(H)$, we can find an open subset $\mathcal{O}(H)\subset\widetilde{\mathcal{O}}(H)$ where the following uniform bounds are satisfied:
\begin{equation}\label{DX_(H+h)}
\sup_{h\in \mathcal{O}(H)\cap C_c^\infty(K)}\sup_{K}\|DX_{H+h}\|<+\infty.
\end{equation}
As $\Sigma$ is of contact type there exists a Liouville vector field $Y$, such that $dH(Y)>0$ on $\Sigma$. In particular, we can assume the existence of an open neighbourhood $\mathcal{U}$ of $\Sigma$ and of a constant $\delta'>0$, such that 
\[
\inf_{U\cap K} dH(Y)>\delta'>0.
\]
By possibly shrinking $\mathcal{O}(H)$, we can assume that for all $h\in \mathcal{O}(H)\cap C_c^\infty(K)$ we have $(H{+}h)^{-1}(0)\cap K \,\subseteq\, U\cap K$ and 
\[
\inf_{U\cap K} d(H{+}h)(Y)>\frac{\delta'}{2}.
\]
Thus we obtain $\displaystyle \qquad\cA^{H+h_n}(v_n,\eta_n)=\eta_n\int_{S^1} d_{v_n}(H{+}h_n)(Y)>\frac{\eta_n \delta'}{2}$.\medskip\\
With (\ref{A->0}), we conclude that $\lim_{n\to \infty}\eta_n=0$. On the other hand, as $(v_n,\eta_n)\in \Crit^+(\cA^{H+h_n})$, we get uniform bounds on the derivative of $v_n$
\[
v_n(t)\in K, \qquad \|\partial_tv_n(t)\|\leq |\eta_n|\sup_{K}\|X_{H+h_n}\|\quad \forall\ n\in \N,\ t\in S^1.
\]
This allows us to use the Arzel{\'a}-Ascoli theorem, which yields a convergent subsequence (which we denote the same)
\[
\lim_{n\to\infty}(v_n,\eta_n)=(x,0)\in \Sigma{\times}\{0\}.
\]
Let $x\in\mathcal{V} \subseteq M,\; \varphi \,\colon \mathcal{V}\to\R^{2m}$ be a coordinate chart around $x$. For $n$ big enough we can assume $v_n(S^1)\subseteq \mathcal{V}$. Then $x_n \coloneqq \varphi \circ v_n$ satisfies the respective Hamiltonian equation $\partial_tx_n=\eta_n D\varphi^{-1}(X_{H+h_n})(x_n)$. Hence $\partial_{tt}x_n=\eta_n D\big(D\varphi^{-1}(X_{H+h_n})\big)(\partial_tx_n)$. However, from Remark \ref{rem:C2small} we can infer that
\[
\eta_n \big\|D\left(D\varphi^{-1}(X_{H+h_n})\right)\big\|\geq 2\pi,
\]
which together with \eqref{DX_(H+h)} contradicts $\lim_{n\to \infty}\eta_n= 0$.
\end{proof}
Fix a strongly tentacular Hamiltonian $H \colon T^*\mathbb{R}^n\to \mathbb{R}$ and let $\mathcal{O}(H)$ be an open neighbourhood of $0$ in $C_c^\infty(T^*\R^n)$, such that all Hamiltonians from $H+\mathcal{O}(H)$ are strongly tentacular and $\mathcal{O}(H)$ is as in Lemma \ref{lem:delta1}. Fix sets $K\subseteq \mathcal{V}\subseteq T^*\R^n$, such that $K\neq \emptyset$ is compact and $\mathcal{V}$ is open and precompact. Let $\delta>0$ be a constant as in Lemma \ref{lem:delta1}; let $\y\in (0,\delta)$ and let $0<\varepsilon_{0},\, \tilde{c}<\infty$ be constants as in \cite[Lem.\ 2.1]{pasquotto2017} depending only on $H$.

Finally, fix $h_\pm\in \mathcal{O}(H)\cap C^\infty_c(K)$ and let 
$$
\mathbb{R}\ni s\longmapsto (H_s,J_s)\in \big(H+\mathcal{O}(H)\big)\times \mathcal{J}(\mathcal{V}{\times}(\R \setminus [-\y,\y]), \mathbb{J}),
$$
 be a smooth homotopy of Hamiltonians and $\omega$-compatible almost complex structures, constant outside of $[0,1]$, such that $H_s=H+h_-$ for $s\leq 0$ and $H_s=H+h_+$ for $s\geq 1$. In this setting we formulate the following lemma:

\begin{lem}\label{lem:delta2}
Let $(v^\pm,\eta^\pm)$ be a pair of critical points of $\cA^{H+h_\pm}$.
If the homotopy $\{H_s,J_s\}_{s\in\R}$ satisfies
\begin{equation}
\|\partial_{s}H_{s}\|_{L^{\infty}}<\min \left\lbrace \frac{\varepsilon_{0}}{8\big(\tilde{c}\varepsilon_0+\|J\|_{L^{\infty}}^{3/2}\big)},\frac{7\delta}{16\tilde{c}(1+\delta)}\right\rbrace,
\label{Hs}
\end{equation}
and if $u$ is a solution to the equation $\partial_s u=\nabla_{J_s}\cA^{H_s}(u)$ with $\displaystyle\lim_{s\to\pm\infty}u(s)=(v^\pm,\eta^\pm)$, then $\cA^{H+h_-}(v^-,\eta^-)>\delta$ implies
$\cA^{H+h_+}(v^+,\eta^+)>\delta$.
\end{lem}
\begin{proof}
The following proof is an adjustment of the proof of \cite[Cor.\ 3.8]{CieliebakFrauenfelder2009} to the setting of strongly tentacular Hamiltonians and is based on the results proven in \cite[Prop.\ 3.3]{pasquotto2017}.
Abbreviate
\[
a \coloneqq  \cA^{H+h_-}(v^-,\eta^-),\qquad b \coloneqq \cA^{H+h_+}(v^+,\eta^+).
\]
By assumption $a\geq \delta$. Let $u=(v,\eta)\in C^{\infty}\big(\mathbb{R}, C^{\infty}(S^{1}, T^*\R^n) {\times} \mathbb{R}\big)$ be a solution to the equation $\partial_s u=\nabla_{J_s}\cA^{H_s}(u)$ with $\lim_{s\to\pm\infty}u(s)=(v^\pm,\eta^\pm)$. Then our setting satisfies the assumptions of \cite[Prop. 3.3]{pasquotto2017}.

First assume additionally $|b|\leq a$. Then $b-a\leq 0$ and \cite[Prop.\ 3.3, eq.\ (3.7)]{pasquotto2017} gives
\begin{align}
\|\eta\|_{L^{\infty}(\mathbb{R})} & \leq \frac{8}{7} \left(\tilde{c}\Big(\max\{|a|, |b|\}+1\Big)+\frac{b-a}{\varepsilon_{0}}\|J\|_{L^{\infty}}^{\frac{3}{2}}\right)\nonumber\\
& \leq \frac{8}{7} \tilde{c}(a+1)\leq a \tilde{c}\frac{8}{7}\left(1+\frac{1}{\delta}\right)\label{eta}
\end{align}
On the other hand, by equation (3.5) from the proof of \cite[Prop. 3.3]{pasquotto2017}, we have that
\begin{equation}
\|\partial_{s}u\|_{L^{2}(\mathbb{R}\times S^{1})}^{2} \leq \|J\|_{L^{\infty}} (b-a + \|\eta\|_{L^{\infty}}\|\partial_{s}H_{s}\|_{L^{\infty}}).
\label{partialU}
\end{equation}
Combined with \eqref{eta} and \eqref{Hs}, we obtain the following estimate:
\[
b\geq a - \|\eta\|_{L^{\infty}(\mathbb{R})}\|\partial_{s}H_{s}\|_{L^{\infty}}\geq a \left(1- \tilde{c}\frac{8}{7}\Big(1+\frac{1}{\delta}\Big)\|\partial_{s}H_{s}\|_{L^{\infty}}\right)\geq \frac{1}{2}a.
\]
In particular, $\cA^{H+h_-}(v^-,\eta^-)=b>0$. By assumption $h_-\in \mathcal{O}(H)\cap C_c^\infty(K)$, hence by Lemma \ref{lem:delta1} we can conclude that $\cA^{H+h_-}(v^-,\eta^-)\geq \delta$. This implies the result provided $|b|\leq a$.

Now assume $b<-a$. Then $b-a\leq 0$ and by \cite[Prop.\ 3.3, eq.\ (3.7)]{pasquotto2017} we have
\begin{align}
\|\eta\|_{L^{\infty}(\mathbb{R})} & \leq \frac{8}{7} \left(\tilde{c}\Big(\max\{|a|, |b|\}+1\Big)+\frac{b-a}{\varepsilon_{0}}\|J\|_{L^{\infty}}^{\frac{3}{2}}\right)\nonumber\\
& \leq \frac{8}{7} \tilde{c}(1-b) \leq -b \tilde{c}\frac{8}{7}\left(1+\frac{1}{\delta}\right),
\end{align}
where the last inequality comes from the assumption $-b\geq a\geq \delta$. Combining it with \eqref{partialU} and \eqref{Hs}, we obtain the following inequality:
\begin{align*}
a \leq b + \|\eta\|_{L^{\infty}(\mathbb{R})}\|\partial_{s}H_{s}\|_{L^{\infty}} \leq b\left( 1- \tilde{c}\frac{8}{7}\Big(1+\frac{1}{\delta}\Big)\|\partial_{s}H_{s}\|_{L^{\infty}}\right)\leq \frac{1}{2}b< -\frac{1}{2}a,
\end{align*}
which contradicts the assumption $a>\delta>0$. That excludes $b<-a$ and proves the lemma.
\end{proof}

\begin{proof}[Proof of Theorem \ref{thm:posRFHtent}:]
Let us fix a Hamiltonian $H \in \mathcal{H}$. By \cite[Lem.\ 8.8]{pasquotto2018} and \cite[Lem.\ 8.9]{pasquotto2018} for every two close enough regular couples $(H{+}h_1,J_1)$, $(H{+}h_2,J_2)$ with $h_1,h_2\in C_c^\infty(T^*\R^n),\ J_1,J_2\in \scrJ_\star$ and a homotopy $\Gamma=\{H{+}h_s,J_s\}_{s\in\R}$ of compactly perturbations $h_s$ of Hamiltonian $H$ and almost complex structures $J_s\in \scrJ_\star$ one can construct a homomorphism 
\[
\phi^\Gamma \colon  CF_*\left(H+h_1,f_1\right)\to CF_*\left(H+h_2,f_2\right),
\]
which is defined by counting the perturbed flow lines with cascades (cf.\ \cite[Sec.\ 11.1]{Audin2014}, \cite[Lem.\ 7.2]{pasquotto2018}). Moreover, $\phi^\Gamma$ satisfies $\phi^\Gamma\circ \partial_1= \partial_2 \circ \phi^\Gamma$, thus induces a homomorphism on the homology level
\[
\Phi^\Gamma \colon RFH(H+h_1,J_1)\to RFH(H+h_2,J_2).
\]
Moreover, by Lemma \ref{lem:delta2} we know that if the two regular couples $(H{+}h_1,J_1),\linebreak[1](H{+}h_2,J_2)$ are close enough, then the preimage of $CF_*^+(H{+}h_2,f_2)$ under $\phi^\Gamma$ lies in $CF_*^+(H{+}h_1,f_1)$. In other words
\begin{equation}\label{phiGamma}
\phi^\Gamma\left(CF^{\leq 0}\left(H+h_1,f_1\right)\right)\subseteq CF^{\leq 0}\left(H+h_2,f_2\right).
\end{equation}
This together with \eqref{filtration} allows us to infer that the restriction $\phi^\Gamma_+$ of $\phi^\Gamma$ to\linebreak[4] $CF_*^+\left(H{+}h_1,f_1\right)$ constructed by counting the perturbed flow lines with cascades between critical points with positive action, commutes with the respective boundary operators $\partial^{\scriptscriptstyle +}_1$ and $\partial^{\scriptscriptstyle +}_2$ and thus induces a homomorphism $\Phi^\Gamma_+ \colon  RFH^+(H{+}h_1,J_1)\to RFH^+(H{+}h_2,J_2)$.

Now, we show that $\Phi^\Gamma_+$ is an isomorphism, using the fact that $\Phi^\Gamma$ is an isomorphism by \cite[Prop.\ 11.2.9]{Audin2014}. Via the inverse homotopy $\Gamma^{-1}=\{H{+}h_{1-s},J_{1-s}\}_{s\in\R}$, one can construct analogously a homomorphism
\[
\phi^{\Gamma^{-1}} \colon  CF_*\left(H+h_2,f_2\right)\to CF_*\left(H+h_1,f_1\right).
\]
As this homomorphism also satisfies $\phi^{\Gamma^{-1}}\circ \partial_2 = \partial_1 \circ \phi^{\Gamma^{-1}}$ and
\begin{equation}
\phi^{\Gamma^{-1}}\left(CF^{\leq 0}\left(H+h_2,f_2\right)\right) \subseteq CF^{\leq 0}\left(H+h_1,f_1\right),\label{phiGamma-1}
\end{equation}
it induces a homomorphism $\Phi^{\Gamma^{-1}}_+ \colon  RFH^+(H{+}h_2,J_2)\to RFH^+(H{+}h_1,J_1)$.

Finally, by \cite[Prop.\ 11.2.9]{Audin2014} there exists a homomorphism
\begin{fleqn}
\begin{align}
S \colon  CF_*\left(H+h_1,f_1\right) &\to CF_{*+1}\left(H+h_1,f_1\right),\notag\\
\text{satisfying}\mspace{100mu} \phi^{\Gamma^{-1}}\circ \phi^\Gamma- \operatorname{Id} & = S\circ \partial_1+\partial_1\circ S\label{S1}\\
\text{and} \mspace{80 mu} S \left(CF^{\leq 0}\left(H+h_1,f_1\right)\right) & \subseteq CF^{\leq 0}\left(H+h_1,f_1\right),\label{S2}
\end{align}
\end{fleqn}
where \eqref{S2} comes from applying Lemma \ref{lem:delta2} once again this time to the flowlines with cascades coming from a homotopy of homotopies (cf.\ \cite[Thm.\ 11.3.11]{Audin2014}). By combining \eqref{filtration}, \eqref{phiGamma}, \eqref{phiGamma-1} and \eqref{S2} we infer that the restriction of $S$ to $CF^+\left(H{+}h_1,f_1\right)$ also satisfies \eqref{S1} with $\phi^\Gamma_+$ and $\phi^{\Gamma^{-1}}_+$, thus establishing that $\Phi^\Gamma_+$ and $\Phi^{\Gamma^{-1}}_+$ are also isomorphisms on the homology level.

This proves that $RFH^+(H)$ is invariant under small enough\footnote{The invariance for any compact perturbation follows by splitting a given perturbation into a sequence of smaller perturbations (see \cite{CieliebakFrauenfelder2009}, p.\ 275)} compactly supported perturbations of $H$ and (via $h_1{=}h_2{=}0$) is invariant of $J$ or $(f,g)$. 
\end{proof}
\section{Tentacular hyperboloids}
\subsection{H{\"o}rmander classification}
Let $A$ be a non-degenerate, quadratic, symmetric matrix and consider the non-degenerate quadratic Hamiltonian $H$ on $(T^*\R^n,\omega_n)$
\begin{equation}
H(x)\coloneqq \frac{1}{2}x^TAx -1.\label{quadHam}
\end{equation}
The hypersurface $\Sigma\coloneqq H^{-1}(0)$ is diffeomorphic to $S^{l-1}\times \R^{2n-l}$, where $(l,2n-l)$ is the signature of $A$. A hyperboloid is the $0$-level set of a quadratic Hamiltonian $H$ as in \eqref{quadHam} with $1\leq l  \leq 2n-1$.

\begin{remark} \label{rem:Lioville}
Note that every hyperboloid is a hypersurface of contact type, as the radial Liouville vector field $Y=\frac{1}{2}x\partial_x$ satisfies $dH_x(Y)=H(x)+1=1\big|_\Sigma>0$. Moreover, the Hamiltonian vector field $X_H$ coincides with the Reeb vector field corresponding to $\iota_Y\omega_n\big|_\Sigma$, as $X_H\in \ker \omega_n\big|_\Sigma$ and $\iota_Y\omega_n(\cdot,X_H)=dH(Y)=1$ on $\Sigma$. 
\end{remark}
In \cite[Sec.\ 9]{pasquotto2018} the last author together with Pasquotto and Vandervorst introduced the notion of \emph{symplectic hyperboloids}, which are equivalence classes of the set of hyperboloids under the action of the linear symplectic group $\operatorname{Sp}(\R^{2n})$. A symplectic hyperboloid is called a \emph{tentacular hyperboloid} if it admits a strongly tentacular representative of the form \eqref{quadHam}. Proposition \ref{prop:TentEx} above provides plenty of examples of tentacular hyperboloids.

According to H{\"o}rmander \cite[Thm.\ 3.1]{Hormander1995} each equivalence class is uniquely determined by the eigenvalues of $\mathbb{J}A$. Observe that if $\lambda$ is an eigenvalue of $\mathbb{J}A$ with a corresponding $m\times m$ block in the Jordan decomposition, then its additive inverse and its complex conjugate are also eigenvalues of $\mathbb{J}A$ each with a $m\times m$ block.

By H{\"o}rmander classification, $T^*\R^n$ splits into a direct sum of symplectic subspaces $S_i$. In other words, after a symplectic change of coordinates we can assume that the matrix $A$ is a diagonal block matrix consisting of matrices $A_i=A|_{S_i}$. Each matrix $A_i$ corresponds to a $m\times m$ block with an eigenvalue $\lambda$ in the Jordan decomposition of $\mathbb{J}A$ and it is determined as follows:
\begin{enumerate}[label=(\alph*)]
\item if $\operatorname{Im}(\lambda)=0$, then $S_i\coloneqq T^*\R^m$ and $A_i$ is a $2m \times 2m$ block matrix $\left(\begin{smallmatrix} 0 & B\\ B^T & 0\end{smallmatrix}\right)$ where $B=\{b_{j,k}\}_{j,k=1}^m$ with 
\[
b_{j,k}\coloneqq \left\lbrace\begin{array}{c l}
|\lambda| &  \quad \textrm{if}\quad j=k,\\
1 & \quad \textrm{if}\quad j= k+1,\\
0 &\quad \textrm{otherwise}.
\end{array}
\right.
\]
Thus $B$ is a $m \times m$ matrix with $|\lambda|$'s on the diagonal and with $1$'s under the diagonal for $m>1$. 
The signature of $A_i$ is $(m,m)$ and the corresponding Hamiltonian $H_i \colon T^*\R^m \to \R$ is given by
\[
H_i(q,p)\coloneqq |\lambda|\sum_{j=1}^mq_jp_j+\sum_{j=1}^{m-1}q_{j+1}p_j.
\]
\item if $\operatorname{Re}(\lambda)\neq 0,\ \operatorname{Im}(\lambda)\neq 0$, then $S_i\coloneqq T^*\R^{2m}$ and
 $A_i$ is an $4m\times 4m$ block matrix $A_i\coloneqq  \left(\begin{smallmatrix} 0 & B\\ B^T & 0\end{smallmatrix}\right)$ where $B=\{b_{j,k}\}_{j,k=1}^{2m}$ with
\[
b_{j,k}\coloneqq \left\lbrace\begin{array}{c l}
|\operatorname{Re}(\lambda)| &\quad \textrm{if}\quad j=k,\\
|\operatorname{Im}(\lambda)| &\quad \textrm{if}\quad \frac{j}{2} \in \mathbb{N} \quad \textrm{and}\quad  k= j-1,\\
-|\operatorname{Im}(\lambda)| &\quad \textrm{if}\quad \frac{k}{2} \in \mathbb{N} \quad \textrm{and}\quad j= k-1,\\
1 &\quad \textrm{if}\quad k=j+2,\\
0 &\quad \textrm{otherwise}.
\end{array}
\right.
\]
The signature of $A_i$ is $(2m,2m)$ and $H_i \colon T^*\R^{2m} \to \R$ is given by
\[
H_i(q,p)\coloneqq   \sum_{j=1}^{2m-2}q_jp_{j+2}+|\operatorname{Re}(\lambda)|\sum_{j=1}^{2m}q_jp_j+|\operatorname{Im}(\lambda)|\sum_{j=1}^m \big( q_{2j}p_{2j-1}-q_{2j-1}p_{2j} \big).
\]

\item if $\operatorname{Re}(\lambda)=0$, then $S_i\coloneqq T^*\R^m$ and for some $\gamma=\pm 1$, $A_i$ is an $2m\times 2m$ block matrix $\left(\begin{smallmatrix} B & 0\\ 0 & B^P\end{smallmatrix}\right)$,
where $B=\{b_{j,k}\}_{j,k=1}^{m}$ with
\[
b_{j,k}\coloneqq \gamma\left\lbrace\begin{array}{c l r}
2|\operatorname{Im}(\lambda)| &\quad \textrm{if}\quad j=k=\frac{m+1}{2}\in \mathbb{N},\\
-2 &\quad \textrm{if}\quad j=k=\frac{m+2}{2}\in \mathbb{N},\\
|\operatorname{Im}(\lambda)| &\quad \textrm{if}\quad j\neq k\quad \textrm{and}\quad j+k=m+1,\\
-1 &\quad \textrm{if}\quad j\neq k\quad \textrm{and}\quad j+k=m+2,\\
0 &\quad \textrm{otherwise}.
\end{array}
\right.
\]
and $B^P$ is the reflection of $B$ with respect to the anti-diagonal. That is, $B^P=\{b_{\delta(j,k)}\}_{j,k=1}^m$ with $\delta(j,k)=(m{+}1\,{-}j, m{+}1\,{-}k)$.
The signature of $A_i$ is $(m,m)$ if $m$ is even and $(m+\gamma,m-\gamma)$ if $m$ is odd, and the corresponding Hamiltonian $H_i \colon T^*\R^m \to \R$ is equal to
\[
H_i(q,p)\coloneqq \frac{\gamma}{2}\left(|\operatorname{Im}(\lambda)|\sum_{j=1}^m \big(q_jq_{m+1-j}+ p_jp_{m+1-j} \big)-\sum_{j=1}^{m-1} \big(q_{j+1}q_{m+1-j} +p_jp_{m-j} \big)\right).
\]
\end{enumerate}

The H{\"o}rmander classification determines a unique (up to permutation of the blocks) representative of the equivalence class of non-degenerate matrices under the action of $\operatorname{Sp}(\R^{2n})$. In particular, after a linear symplectic change of coordinates, every quadratic Hamiltonian of the form \eqref{quadHam} can be represented as a sum of Hamiltonians $H_i \colon S_i\to\R,$
\[
H(x)=\sum_iH_i(x_i)-1, \qquad H_i(x_i)\coloneqq  \frac{1}{2} x_i^TA_ix_i, \qquad x_i\in S_i,
\]
with $A_i$ either of type (a), (b) or (c).
\begin{remark}
Observe, that the Hamiltonian $H_i(x_i)=\frac{1}{2} x_i^TA_ix_i$ for a matrix $A_i$ either of type (a) or (b) vanishes on the $0$-section $\chi_0=\R^{m}\times\{0\}\subseteq T^*\R^m =S_i$. In other words, if a quadratic Hamiltonian $H \colon T^*\R^m\to \R$ as in \eqref{quadHam} comes from a non-degenerate, symmetric matrix $A$ and all the eigenvalues of the matrix $\mathbb{J}_m A$ have non-zero real part, then there exists a Lagrangian subspace of $T^*\R^m$ on which the Hamiltonian $H+1$ vanishes.
\end{remark}

A classical result \cite[Lem.\ 2.43]{mcduff2005} shows that a positive definite matrix is simplectically diagonalizable. More precisely, by applying a linear symplectic change of coordinates to a positive definite matrix, we can obtain a diagonal matrix with couples of positive real numbers $\mu_1,\mu_1,\dots, \mu_k,\mu_k,\ \mu_j\leq \mu_{j+1}$ on the diagonal. In the words of the H{\"o}rmander classification: a positive definite matrix corresponds to $k$ block matrices of type (c) and dimension $2$. Therefore, we can conclude that if $A$ is a positive definite matrix, then all the eigenvalues of $\mathbb{J}A$ are purely imaginary and equal $\pm i\mu_j,\ j=1,\dots, k$.

Given $x \in \mathbb{R}^n$, write $x = (x',x'')$ with respect to the splitting $\mathbb{R}^n = \mathbb{R}^k \oplus \mathbb{R}^{n-k}$. Define the symplectic splitting $T^*\R^n\cong T^*\R^k \times T^*\R^{n-k}$ via the symplectomorphism 
\begin{equation}\label{splitting}
\begin{aligned}
\sigma \colon \big( T^*\R^n, \omega_n \big) &\to \big( T^*\R^k \times T^*\R^{n-k}, \omega_k \oplus \omega_{n-k} \big)\\
 \sigma (q,p)  &\coloneqq  \big( (q',p'), (q'',p'') \big).
\end{aligned}
\end{equation}
Then $\sigma(\Sigma)\cap\left(T^*\R^k\times\{0\}\right)=\Sigma_0\times\{0\}$. 

\begin{remark}\label{H2}
From the arguments presented above we can deduce the following properties of the matrices and Hamiltonians satisfying \eqref{defH}:
\begin{enumerate}[label=(\roman*)]
\item The eigenvalues of $\mathbb{J}_kA_0$ are purely imaginary;
\item Without loss of generality we can assume that $A_0$ is a diagonal matrix with couples of positive real numbers $\mu_1,\mu_1,\dots, \mu_k,\mu_k,\ \mu_j\leq \mu_{j+1}$ on the diagonal, which correspond to eigenvalues $\pm i\mu_j$ of $\mathbb{J}_kA_0$;
\item The matrix $A_1$ has signature $(n-k,n-k)$;
\item Define $\Sigma_0\coloneqq H_0^{-1}(0), \Sigma\coloneqq H^{-1}(0)$ and 
\begin{equation}\label{Sigma1}
\Sigma_1\coloneqq \Sigma \cap \sigma^{-1}(T^*\R^k\times S),
\end{equation}
where $S$ denotes the $n{-}k$-dimensional subspace of $T^*\R^{n-k}$ spanned by the eigenvectors of $A_1$ corresponding to positive eigenvalues. Then there exist diffeomorphisms $\phi_0 \colon \Sigma_0\rightarrow S^{2k-1}$ and $\phi \colon \Sigma\rightarrow S^{n+k-1}\times \R^{n-k}$, such that $\phi(\Sigma_1)=S^{n+k-1}\times\{0\}$. 
\item\label{DefnS} Without loss of generality, we can assume that the Hamiltonian $H_1$ satisfies $H_1(x, 0)= 0$ for all $x\in \R^{n-k}$, i.e.\ $H_1$ vanishes on the $0$-section $\chi_0=\R^{n-k}\times\{0\}\subseteq T^*\R^{n-k}$.
In particular,
$$
\sigma^{-1}(\Sigma_0 \times \chi_0) \subseteq \Sigma.
$$
\item For $S$ defined as above we have $S\cap \chi_0=\{0\}$.
\end{enumerate}
\end{remark}
\begin{remark}
We emphasise (cf. Remark \ref{works-for-any-tent}) that all points of Remark \ref{H2} only use that $H$ is of the form \eqref{defH} with $A_0$ positive definite and $\mathbb{J}A_1$ hyperbolic.
\end{remark}

The embedding $T^*\R^k\cong T^*\R^k\times \{0\}^{2(n-k)}\hookrightarrow T^*\R^n$ induces an inclusion
\begin{equation}
\label{eqj}
j \colon \Sigma_0 \hookrightarrow \Sigma_1,
\end{equation}
and a retraction $r \colon \Sigma \to \Sigma_1$. Recall that the Umkehr map $j_! \colon H_*(\Sigma_1) \to H_{* +k-n}( \Sigma_0)$ is defined as the composition 
\[
j_! \coloneqq ( - \frown [\Sigma_0] ) \circ j^* \circ ( - \frown [\Sigma_1])^{-1}. 
\]
For later purposes, we observe that the composition 
\begin{equation}
    \label{umkehr}
    j_! \circ r_* \colon H_*( \Sigma) \to H_{* + k - n}(\Sigma_0),
\end{equation}
is an isomorphism for $* = n +k - 1$ and zero otherwise.
\subsection{Periodic orbits}
In this subsection we discuss the properties of non-degenerate periodic orbits of quadratic Hamiltonians $H(x,y)=H_0(x)+H_1(y)$ on $T^\ast\R^n=T^\ast\R^k\oplus T^\ast\R^{n-k}$ satisfying \eqref{defH}. In particular, we establish an action and degree-preserving 1-to-1 correspondence between periodic orbits of $X_H$ on $\Sigma=H^{-1}(0)$ and periodic orbits of $X_{H_0}$ on $\Sigma_0=H^{-1}_0(0)$.

\begin{lem}\label{lem:orbits}
There is an action preserving $1$-to-$1$ correspondence of the periodic orbits of $X_{H}$ on $\Sigma$ with the periodic orbits of $X_{H_0}$ on $\Sigma_0$. Explicitly, for $\eta\neq 0$
\[
v\colon S^1\to \Sigma,\qquad \partial_t v=\eta X_H(v), 
\]
if and only if $ \sigma(v)=(v_0,0)$, where
\[
v_0\colon S^1 \to \Sigma_0, \qquad \partial_t v_0=\eta X_{H_0}(v_0).
\]
Moreover $\mathcal{A}^H(v)=\mathcal{A}^{H_0}(v_0)=\eta$, for the Rabinowitz action functionals of $H$ on $(T^*\R^n,\omega_n)$ and $H_0$ on $(T^*\R^k,\omega_k)$ respectively. 
\end{lem}
\begin{proof}
For a Hamiltonian $H$ as in \eqref{quadHam} the Hamiltonian flow of $X_H$ is given by
\[
\phi^t (x)=\exp(t \mathbb{J}_n A)x.
\]
Consequently, $(v,\eta)\in \Crit(\cA^H)$ if and only if $v(0)\in \Sigma$ and $\phi^\eta(v(0))=v(0)$ or, equivalently, $v(0)\in \ker \! \big( \! \exp(\eta \mathbb{J}_n A)-\operatorname{Id} \! \big)$.  By assumption the matrix $A$ can be presented with respect to the symplectic splitting \eqref{splitting} as a block matrix:
\[
A\coloneqq \left(\begin{array}{c c}
A_0 & 0\\
0 & A_1
\end{array}
\right),
\]
where $A_0$ and $A_1$ are as in \eqref{defH}. As a result,
\[
\ker \! \big(\! \exp(\eta \mathbb{J}_n A)-\operatorname{Id}\!\big)=\ker\!\big(\!\exp(\eta \mathbb{J}_k A_0)-\operatorname{Id}\!\big)\oplus \ker \! \big( \! \exp(\eta \mathbb{J}_{n-k} A_1)-\operatorname{Id} \! \big).
\]
A simple calculation shows that $\det\! \big( \! \exp(\eta \mathbb{J}_{n-k} A_1)-\operatorname{Id}\!\big)=0$ if and only if $e^{\eta\mu}=1$ for some eigenvalue $\mu$ of $\mathbb{J}_{n-k}A_1$. However, by assumption all the eigenvalues of $\mathbb{J}_{n-k}A_1$ have non-zero real parts, hence $e^{\eta\mu}\neq 1$ for all $\eta\neq 0$.

Consequently, $(v,\eta)\in \Crit(\cA^H)$ and $\eta\neq 0$ if and only if $v(0)\in \Sigma$ and $\sigma \circ v(0)=(w_0,0)$ for some $w_0\in \ker\! \big( \! \exp(\eta \mathbb{J}_k A_0)-\operatorname{Id}\! \big)$. If we define $v_0(t)\coloneqq \phi^{\eta t}(w_0)$ then $(v_0,\eta)\in \Crit(\cA^{H_0})$, as $\sigma(\Sigma)\cap\left(T^*\R^{k}{\times}\{0\}\right)=\Sigma_0\times\{0\}$.

 Finally, recall from Remark \ref{rem:Lioville} that $Y=\frac{1}{2}x\partial_x$ is a Liouville vector field for $\omega_n$ and  satisfies $dH_x(Y)=1$ for all $x\in\Sigma$.
Hence, we have for $(v,\eta)\in \Crit\left(\cA^H\right)$
\begin{equation}\label{action}
\cA^{H}(v,\eta)=\int \lambda_n(\partial_tv)=\eta\int\omega_n(Y,X_H)=\eta\int dH(Y)=\eta,
\end{equation}
where $\lambda_n\coloneqq \iota_Y\omega_n$ is the Liouville form for $Y$. The same arguments for $Y_0=\frac{1}{2}x\partial_x$ on $T^\ast \R^k$ show that $\cA^{H_0}(v_0,\eta)=\eta$. This concludes the proof.
\end{proof}
\begin{cor}\label{cor:CritVal}
Denote by $\{\pm i \mu_l\}_{l=1}^k,\ \mu_l>0$ the eigenvalues of $\mathbb{J}_k A_0$. Then
\[
\operatorname{CritVal}\left(\cA^{H}\right)= \operatorname{CritVal}\left(\cA^{H_0}\right)=\bigcup_{l=1}^k \frac{2\pi }{\mu_l}\Z .
\]
\end{cor}
\begin{proof}
By Lemma \ref{lem:orbits} it suffices to prove the second equality. By the proof of Lemma \ref{lem:orbits}, we know that $(v_0,\eta)\in \Crit\left(\cA^{H_0}\right)$ if and only if
\begin{equation}\label{CritEta}
v_0(t)=\exp(t\eta \mathbb{J}_kA_0)w_0\quad \textrm{and}\quad  w_0\in \Sigma_0\cap \ker \! \big(\! \exp(\eta \mathbb{J}_kA_0)-\operatorname{Id} \! \big).
\end{equation}
As $\ker\big(\exp(\eta \mathbb{J}_kA_0)-\operatorname{Id}\big)\neq \{0\}$ if and only if there exists an eigenvalue $\pm i \mu_l$ of $\mathbb{J}_kA_0$, such that $\eta \mu_l \in 2\pi \Z$, and as $\mathcal{A}^{H_0}(v_0,\eta)=\eta$, the corollary follows.
\end{proof}

\begin{remark}\label{rem:orbits} 
\label{rem:Sk}
We define the following subspaces of the critical sets
\begin{equation}\begin{aligned}
\Lambda_0^\eta &\coloneqq  \left\lbrace (v,\mathfrak{y})\in \Crit\left(\cA^{H_0}\right)\ \big|\ \mathfrak{y}=\eta\right\rbrace,\\
\Lambda^\eta &\coloneqq  \left\lbrace (v,\mathfrak{y})\in \Crit\left(\cA^{H}\right)\ \big|\ \mathfrak{y}=\eta\right\rbrace.
\end{aligned} \label{LambdaEta}
\end{equation}
In particular $\Lambda_0^0= \Sigma_0\times\{0\}$ and $\Lambda^0=\Sigma\times\{0\}$.
Meanwhile for $\eta \neq 0$, we have by Lemma \ref{lem:orbits} $\cA^H(\Lambda^\eta)=\cA^{H_0}(\Lambda_0^\eta)=\eta$ and
\[
\Lambda^\eta= \left\lbrace (\sigma^{-1}(v_0,0), \eta)\ \big|\ (v_0, \eta) \in \Lambda_0^\eta\right\rbrace\cong \Lambda_0^\eta.
\]
As $\Sigma_0$ is diffeomorphic to $S^{2k-1}$, we obtain from \eqref{CritEta} for $\eta \in \operatorname{CritVal} (\cA^{H_0})$ that $\Lambda_0^\eta$ is diffeomorphic to $S^{2m-1}$, where $m\coloneqq  \frac{1}{2}\dim \ker \! \big( \! \exp(\eta \mathbb{J}_kA_0)-\operatorname{Id} \! \big)$.
\end{remark}

Next, we show that the correspondence between the periodic orbits of $X_H$ on $\Sigma$ and the periodic orbits of $X_{H_0}$ on $\Sigma_0$ is not only action- but also degree-preserving.

\begin{prop}\label{prop:muCZ}
If $(v,\eta), \eta\neq 0$ is a periodic orbit of $X_{H}$ on $\Sigma\subseteq T^*\R^n$ with $\sigma(v)=(v_0,0)$ and $(v_0,\eta)$ a periodic orbit of $X_{H_0}$ on $\Sigma_0\subseteq T^*\R^n$, then 
\[
\tCZ(v,\eta)=\tCZ (v_0,\eta).
\]
\end{prop}
Here $\tCZ$ denotes the transverse Conley-Zehnder index, which we will denote in the following proof by $\CZ^\xi$, as it is only evaluated on the contact structure $\xi$.

\begin{proof}
Let $H$ be a quadratic Hamiltonian as in \eqref{quadHam} and $H^{-1}(0)=\Sigma$.
The radial Liouville vector field $Y=\frac{1}{2}x\partial_x$ defines a contact structure $\xi=\ker\left( \iota_{Y}\omega\big|_{T\Sigma}\right)$ on $\Sigma$ by Remark \ref{rem:Lioville}. Writing $L^\omega$ for the symplectic complement of a subspace $L$, we have for $x \in \Sigma $:
\begin{align*}
\xi_x & =\left\lbrace v\in T_x\Sigma\, |\, \omega_n(Y, v)=0 \right\rbrace\\
& = \left\lbrace v\in T_xT^\ast\R^n\, |\, \omega_n(Y, v)=0, dH(v)=0 \right\rbrace
 = \left( \operatorname{span}\{ X_H,Y\}\right)^{\omega}.
\end{align*}
Consequently, $\xi^{\omega}=\operatorname{span}\{ X_H,Y\}$. If $\mathcal{L}_{X_H}$ denotes the Lie-derivative along $X_H$, we obtain by Cartan's formula and Remark \ref{rem:Lioville} that
\[
 \mathcal{L}_{X_H} \iota_Y \omega = d(\iota_{X_H} \iota_Y \omega_n)+\iota_{X_H} d(\iota_Y \omega_n) =d(dH(Y))+\iota_{X_H}\omega_n = d(H{+}1)-dH=0.
\]
Consequently, the Hamiltonian flow $\phi^t$ of $X_H$ preserves $\iota_Y \omega_n$. As $\phi^t$ preserves also $dH$ and $\omega_n$, we find that the vector fields $X_H, Y$, the bundles $\xi, \xi^\omega$ and  the splitting $T(T^*\R^n)=T\mathbb{R}^{2n}=\xi\oplus\xi^\omega$ over $\Sigma$ are also preserved by $\phi^t$.  

Let $\gamma$ be a closed characteristic on $\Sigma$. Then we have 
\[
\gamma^*T\mathbb{R}^{2n}=\gamma^*\xi \oplus \gamma^*\xi^{\omega}.
\]
As the Hamiltonian flow preserves this splitting, it follows by the product property \cite{Robbin1993} that
the Conley-Zehnder index of $\gamma$ is the sum:
\[
\CZ^{T\mathbb{R}^{2n}}(\gamma)=\CZ^{\xi}(\gamma)+\CZ^{\xi^{\omega}}(\gamma).
\]
As $(X_H,Y)\circ \gamma$ provides a trivialisation of $\gamma^*\xi^{\omega}$ and as $\phi^t$ preserves $(X_H,Y)$, we find that $D\phi^t|_{\xi^\omega}$ expressed in this trivialisation is a path of identity matrices and hence
\[
\CZ^{\xi^{\omega}}(\gamma)=\CZ(\operatorname{Id})=0\qquad \Longrightarrow\qquad \CZ^\xi(\gamma)=\CZ^{T\R^{2n}}(\gamma).
\]
The Hamiltonian flow of a Hamiltonian as in \eqref{quadHam} is given by
\[
\phi^t (x)=\exp(t \mathbb{J}_n A)x\quad\textrm{and}\quad D\phi^t=\exp(t \mathbb{J}_n A).
\]
For a closed characteristic $\gamma$ with period $\eta$ we can take $\Phi(t)=\operatorname{Id}$ for all $t$ as the trivialization $\Phi\colon [0,\eta]\times \R^{2n}\to \gamma^* (T\R^{2n})$. With respect to this trivialization 
\[
\CZ^\xi(\gamma)=\CZ^{T\R^{2n}}(\gamma)=\CZ(\exp(t \mathbb{J}_n A)).
\]
Naturally, $\mathbb{R}^{2n}=\mathbb{R}^{2k}\oplus\mathbb{R}^{2(n-k)}$ is a splitting into symplectic subspaces. By assumption, the matrix $A$ can be written as a block matrix with respect to this splitting:
\[
A\coloneqq \left(\begin{array}{c c}
A_0 & 0\\
0 & A_1
\end{array}
\right),
\]
where $A_0$ and $A_1$ are as in \eqref{defH}. This way the Conley-Zehnder index of $\gamma$ is a sum:
\[
\CZ^\xi(\gamma)=\CZ \big(\! \exp(t \mathbb{J}_k A_0) \big)+\CZ \big(\! \exp(t \mathbb{J}_{n-k} A_1) \big).
\]
By Lemma \ref{lem:orbits}, $\gamma$ is also a characteristic on $\Sigma_0\coloneqq H_0^{-1}(0)$. Repeating the arguments presented above yields that $\xi_0\coloneqq \ker(\iota_Y\omega_k)$ is a contact structure on $\Sigma_0$ and that
\[
\CZ^{\xi_0}(\gamma)=\CZ\big(\!\exp(t \mathbb{J}_k A_0) \big).
\]
Thus it suffices to show that $\CZ\big(\! \exp(t \mathbb{J}_{n-k} A_1)\big)=0$. The Conley-Zehnder index of the path of matrices $\exp(t \mathbb{J}_{n-k} A_1), t\in[0, \eta]$ is by definition
\begin{gather}
\CZ \big(\! \exp(t \mathbb{J}_{n-k} A_1) \big) =\frac{1}{2}\operatorname{sgn}(A_1)+\frac{1}{2}\operatorname{sgn}\!\left(A_1\Big|_{\ker \! \big(\! \! \exp(\eta \mathbb{J}_{n-k} A_1)-\operatorname{Id}\! \big)}\right)\nonumber\\
+\sum_{\substack{t\in(0,\eta) ,\\ \det(\exp(t\mathbb{J}_{n-k}A)-\operatorname{Id})=0}}\operatorname{sgn}\!\left(A_1\Big|_{\ker \! \big( \! \exp(t\mathbb{J}_{n-k}A_1)-\operatorname{Id} \! \big)}\right)\label{CZwzor}
\end{gather}
First observe that the first term vanishes, since by assumption $A_1$ has signature $0$.
To analyse the rest of the terms, we will calculate the \emph{crossings}, i.e. those $t\in \R$, such that $\det \! \big(\! \exp(t\mathbb{J}_{n-k}A_1)-\operatorname{Id} \! \big)=0$. A simple calculation shows that $t$ is a crossing of $\exp(t\mathbb{J}_{n-k}A_1)$ if and only if $e^{t\mu}=1$ for some eigenvalue $\mu$ of $\mathbb{J}_{n-k}A_1$. However, by assumption all the eigenvalues of $\mathbb{J}_{n-k}A_1$ have non-zero real parts, hence $e^{t\mu}\neq 1$ for all $t\neq 0$. Therefore, $\exp(t\mathbb{J}_{n-k}A_1)$ has no other crossings than $0$ and all terms in \eqref{CZwzor} vanish giving $\CZ\big(\! \exp(t \mathbb{J}_{n-k} A_1) \big)=0$.
\end{proof}
\subsection{Morse-Bott property}\label{sec:Morse-Bott}
In this section we show that the Morse-Bott conditions \eqref{eq:MorseBott} and \eqref{def:MBT} are satisfied for Hamiltonians $H$ and $H_0$ as in \eqref{defH}.
\begin{lem}
For Hamiltonians $H_0$ and $H$ on $(T^\ast\R^k,\omega_k)$ or $(T^\ast\R^n,\omega_n)$ as in \eqref{defH} the periodic orbits of $X_{H_0}$ and $X_H$ are of Morse-Bott type and
the Rabinowitz action functionals $\cA^{H_0}$ and $\mathcal{A}^H$ are Morse-Bott.
\end{lem}
\begin{proof}
We only give the proof for $H$, the proof for $H_0$ is analogous.

First, we show that $\Crit(\cA^H)$ is a discrete union of connected manifolds $\Lambda$. Let $A=\big(\begin{smallmatrix}A_0 &\\ & A_1\end{smallmatrix}\big)$ be the block matrix such that $H(x,y)=\big(\begin{smallmatrix}x\\y\end{smallmatrix}\big)^TA\big(\begin{smallmatrix}x\\y\end{smallmatrix}\big)-1$. Let $\mu_j$ be the eigenvalues of the matrix $\mathbb{J}_kA_0$ and let $\Lambda^\eta$ be the set of pairs $(v,\eta)$ such that $v$ is an $\eta$-periodic orbit of $X_H$ (see \eqref{LambdaEta}). Recall that $\Lambda^\eta$ is diffeomorphic to a sphere $S^{2m-1}$ or $\Sigma$ and that we have by Corollary \ref{cor:CritVal} and Remark \ref{rem:orbits} that
$$
\Crit\left(\cA^H\right)= \bigcup_{\eta \in \bigcup_{j=1}^k\frac{2\pi}{\mu_j}\Z} \Lambda^\eta.
$$
Next, recall from Remark \ref{rem:Lioville} that quadratic Hamiltonians $H$ are defining for the hypersurface $\Sigma=H^{-1}(0)$, as the Liouville vector field $Y=\frac{1}{2}x\partial_x$ satisfies $dH(Y)|_\Sigma=1$. This allows us to apply \cite[Lem.\ 20]{Fauck2016} and conclude that $\cA^H$ is Morse-Bott if the periodic orbits of $X_H$ are of Morse-Bott type.

To prove this last property consider the projection $(v,\eta)\mapsto v(0)$ and for $\eta\in \Crit(\cA^H)$ denote by $\mathcal{N}^\eta$ the image of $\Lambda^\eta$ under this projection. It suffices to show for all $\eta$ that $\mathcal{N}^\eta\subseteq\Sigma$ is a closed submanifold and that for $p\in\mathcal{N}^\eta$ it holds that
\[T_p\mathcal{N}^\eta=\ker\big(d_p H\big)\cap\ker \! \big(D_p\phi^\eta-\operatorname{Id} \! \big),\]
where $\phi^\eta$ denotes the time $\eta$ flow of $X_H$. Both conditions follow from \eqref{CritEta}:
\begin{align*}
\mathcal{N}^\eta =\pi_{\mathcal{A}^H}(\Lambda^\eta) &= \equalto{\Sigma}{H^{-1}(0)}\cap\ker\big(\underbrace{\exp(\eta \mathbb{J}_nA)}_{=D\phi^\eta}-\operatorname{Id} \! \big)\\
\Longrightarrow\qquad T_p\mathcal{N}^\eta &=\ker (d_pH)\cap\ker \! \big( D_p\phi^\eta-\operatorname{Id} \! \big).\qedhere
\end{align*}
\end{proof}
\section{The hybrid problem}
In the previous section we have shown that there is a $1$-to-$1$, action and degree preserving correspondence between periodic orbits of $X_H$ on $\Sigma$ and of $X_{H_0}$ on $\Sigma_0$, which will allows us to define a chain map 
\[
\psi \colon CF_*(H,f)\to CF_*(H_0,f_0).
\]
However, to ensure that $\psi$ induces a homomorphism
\[
\Psi \colon RFH_\ast(H)\to RFH_\ast(H_0),
\]
we need to verify that it commutes with the boundary operators, which is the subject of this section.

 Let $\Lambda_0\subseteq \Crit(\mathcal{A}^{H_0})$ and $\Lambda\subseteq \Crit(\mathcal{A}^{H})$ be connected components and fix $J_0\in \scrJ_\star^k$ and $J \in \scrJ_\star^n$, such that the pairs $(H_0,J_0)$ and $(H,J)$ are regular in the sense of Definition \ref{def:regular}. To construct $\psi$, we consider the following moduli spaces of pairs of half Floer trajectories of the Hamiltonians $H$ and $H_0$:

\begin{definition} An element of $ \mathcal{M}_{\operatorname{hyb}}(\Lambda_0 , \Lambda)$ is a pair $(u_0,u)$, such that ${u_0 = (v_0, \eta_0)}$ and ${u = (v, \eta)}$ with
\begin{align*}
v_0 & \colon (- \infty ,0] \times S^1 \to T^*\R^k,  && \eta_0 \colon (- \infty ,0] \to \mathbb{R},\\
v & \colon [0, \infty) \times S^1 \to T^*\R^n,  && \eta \colon [0, \infty ) \to \mathbb{R},
\end{align*}
satisfying the Rabinowitz Floer equations
\begin{equation}\label{Floer}
\partial_s u_0 - \nabla_{J_0} \mathcal{A}^{H_0}(u_0) = 0,\qquad
\partial_s u - \nabla_{J} \mathcal{A}^{H}(u) = 0,
\end{equation}
together with the limit conditions
\begin{equation}
  \lim_{s\to -\infty}u_0(s)\in \Lambda_0, \qquad \lim_{s\to +\infty}u(s)\in \Lambda, \label{limit1}
\end{equation} 
and the coupling conditions
\begin{equation}
 \sigma ( v(0,t)) = \big(v_0(0,t),(*,0_{n-k}) \big), \qquad \eta(0) = \eta_0(0), \label{coupling}
\end{equation}
where $\sigma$ is the symplectomorphism defined in \eqref{splitting}.
\end{definition}
Below we will analyse the moduli space $\mathcal{M}_{\operatorname{hyb}}(\Lambda_0 , \Lambda)$. In subsection \ref{ssec:stationary} we explore what happens when $\mathcal{A}^{H}(\Lambda)=\mathcal{A}^{H_0}(\Lambda_0)$; in subsection \ref{ssec:bounds} we prove uniform $L^\infty$-bounds on the elements of $\mathcal{M}_{\operatorname{hyb}}(\Lambda_0 , \Lambda)$ and in subsection \ref{ssec:index} we show that $\mathcal{M}_{\operatorname{hyb}}(\Lambda_0 , \Lambda)$ has the structure of a smooth manifold.

\subsection{Stationary solutions}\label{ssec:stationary}
 For $(u_0,u)\in  \mathcal{M}_{\operatorname{hyb}}(\Lambda_0 , \Lambda)$ we define its energy as
\begin{equation}\label{rem:Enabla}
E(u_0,u)\coloneqq \int_{-\infty}^0\|\partial_su_0\|^2ds+\int_0^{\infty}\|\partial_su\|^2ds\geq0.
\end{equation}
\begin{lem} \label{energy}
For $(u_0,u)\in  \mathcal{M}_{\operatorname{hyb}}(\Lambda_0 , \Lambda)$, its energy satisfies
\[E(u_0,u)=\mathcal{A}^{H}(\Lambda)-\mathcal{A}^{H_0}(\Lambda_0).\]
In particular, if $\mathcal{M}_{\operatorname{hyb}}(\Lambda_0 , \Lambda)\neq \emptyset$, then $\mathcal{A}^{H}(\Lambda)\geq\mathcal{A}^{H_0}(\Lambda_0)$.
\end{lem}
\begin{proof}
At first, we have by \eqref{Floer} and \eqref{limit1} that
\begin{align*}
E(u_0,u) &= \int_{-\infty}^0\|\nabla_{J_0} \cA^{H_0}(u_0(s))\|^2ds+\int_0^{\infty}\|\nabla_J \cA^{H}(u(s))\|^2ds\\
&=\int_{-\infty}^0 \frac{d}{ds}\cA^{H_0}(u_0(s)) ds + \int_0^{\infty}\frac{d}{ds}\cA^{H}(u(s)) ds\\
&=\mathcal{A}^{H_0}(u_0(0))-\mathcal{A}^{H_0}(\Lambda_0)+\mathcal{A}^{H}(\Lambda)-\mathcal{A}^{H}(u(0)).
\end{align*}
Secondly, we have for $(v_1,\eta)\in C^{\infty}(S^1,\chi_0)\times \mathbb{R}$ that
\[
\mathcal{A}^{H_1}(v_1,\eta)= \int_{S^1} \lambda(\partial_t v_1)-\eta \int_{S^1} H_1(v_1)=0.
\]
The first integral vanishes, as the primitive $\lambda=pdq$ of $\omega_{n-k}$ vanishes on the $0$-section $\chi_0\coloneqq \R^{n-k}\times\{0\}\subseteq T^*\R^{n-k}$. The second integral vanishes, as $H_1$ vanishes on $\chi_0$ by \ref{DefnS} of Remark \ref{H2}. Using \eqref{coupling}, we have $u(0)=(\sigma^{-1}(v_0(0),v_1(0)),\eta)$ and $u_0(0)=(v_0(0),\eta)$ with $v_1(0)\in C^\infty(S^1,\chi_0)$, and therefore we find
\[
\mathcal{A}^{H}(u(0)) = \mathcal{A}^{H_0}(u_0(0))\qquad\Longrightarrow\qquad E(u_0,u)=\mathcal{A}^{H}(\Lambda)-\mathcal{A}^{H_0}(\Lambda_0).\qedhere
\]
\end{proof}
We formulate the next lemma using the notation from \eqref{LambdaEta}:
\begin{lem}\label{lem:stationary}
For $\eta \in \operatorname{CritVal}(\cA^{H_0})= \operatorname{CritVal}(\cA^H)$, $\eta\neq 0$ the moduli space $\mathcal{M}_{\operatorname{hyb}}(\Lambda_0^\eta , \Lambda^\eta)$ consists of stationary solutions, of the form
\begin{equation*}\begin{aligned}
u_0(s,t) &\coloneqq  (v(t), \eta), && \forall\ (s,t)\in (-\infty, 0]\times S^1,\\
u(s,t) & \coloneqq  (\sigma^{-1}(v(t), 0), \eta), && \forall\ (s,t)\in [0,+\infty)\times S^1,
\end{aligned} \qquad  \text{where }(v,\eta)\in \Lambda_0^\eta.
\end{equation*}
For $\eta=0$, $\mathcal{M}_{\operatorname{hyb}}(\Lambda_0^0 , \Lambda^0)$ consists of stationary solutions of the form
\begin{equation*}\begin{aligned}
u_0(s,t) & \coloneqq  (x, 0), && \forall\ (s,t)\in (-\infty, 0]\times S^1,\\
u(s,t) & \coloneqq  (\sigma^{-1}(x,y), 0), && \forall\ (s,t)\in [0,+\infty)\times S^1,
\end{aligned} \qquad  \text{where } (x,y)\in \Sigma_0 \times \chi_0.
\end{equation*}
\end{lem}
\begin{proof}
If $\mathcal{A}^{H}(\Lambda)=\mathcal{A}^{H_0}(\Lambda_0)$, then $E(u_0,u)=0$ by Lemma \ref{energy} for every $(u_0,u)\in\mathcal{M}_{\operatorname{hyb}}(\Lambda_0 , \Lambda)$. As a result $\mathcal{M}_{\operatorname{hyb}}(\Lambda_0 , \Lambda)$ consists of stationary solutions of the form
\begin{align*}
u_0(s,t) & \coloneqq  (v_0(t), \eta), && \forall\ (s,t)\in (-\infty, 0]\times S^1,\\
u(s,t) & \coloneqq  (\sigma^{-1}(v_0(t), v_1(t)), \eta), && \forall\ (s,t)\in [0,+\infty)\times S^1,
\end{align*}
where by \eqref{limit1} and \eqref{coupling} it holds
\[
(v_0,\eta) \in \Lambda_0, \qquad (\sigma^{-1}(v_0,v_1),\eta)\in \Lambda, \qquad v_1(t)\in \chi_0 \qquad \forall\ t\in S^1.
\]
On the other hand, by Lemma \ref{lem:orbits} we know that if $\mathcal{A}^{H}(\Lambda)=\mathcal{A}^{H_0}(\Lambda_0) \neq 0$, then
\[
\Lambda= \left\lbrace (\sigma^{-1}(v_0,0),\eta)\ \big|\ (v_0,\eta)\in \Lambda_0\right\rbrace.
\]
Combining the two facts above proves the first claim. Similarly, we obtain the second claim by observing that $\Sigma_0\times \chi_0\subseteq \Sigma_0\times H^{-1}_1(0)\subseteq \sigma(\Sigma)$.
\end{proof}
\subsection{Bounds}\label{ssec:bounds}

One of the crucial steps in constructing the homomorphism between $RFH(H)$ and $RFH(H_0)$ is to establish $L^\infty$-bounds on $\mathcal{M}_{\operatorname{hyb}}(\Lambda_0,\Lambda)$.

To formulate the $L^\infty$-bounds in Proposition \ref{prop:hybBound} and Lemma \ref{lem:compact} we introduce the following notation: for a compact subset $N\subseteq \Sigma$ and a connected component $\Lambda_{0} \subseteq \Crit(\cA^{H_{0}})$ we denote 
\begin{align*}
\mathcal{C}({\cA^{H}},N) & \coloneqq  \left\lbrace x\in \Crit(\mathcal{A}^H)\ \Big|\ |\mathcal{A}^H(x)|>0 \quad \textrm{or}\quad x\in N\times\{0\}\right\rbrace,
\label{eqn:CAHY}\\
\mathcal{N}_{\operatorname{hyb}}(\Lambda_0, N) & \coloneqq \left\lbrace (u_0,u) \in \mathcal{M}_{\operatorname{hyb}}(\Lambda_0,\Sigma{\times}\{0\})\ \Big|\ \lim_{s\to +\infty}u(s)\in N \right\rbrace,
\end{align*}
For a pair of connected components $(\Lambda_{0},\Lambda) \subseteq \Crit(\cA^{H_{0}}) \times\left(\Crit(\cA^H)\big\backslash (\Sigma{\times}\{0\})\right)$ we denote $\mathcal{N}_{\operatorname{hyb}}(\Lambda_0, \Lambda)\coloneqq \mathcal{M}_{\operatorname{hyb}}(\Lambda_0, \Lambda)$.

\begin{prop}\label{prop:hybBound}
Consider a compact subset $N\subseteq \Sigma$ and a pair of connected components $(\Lambda_{0},\Lambda) \subseteq \Crit(\cA^{H_{0}}) \times\mathcal{C}(\cA^{H},N)$,
such that $a \leq \cA^{H_{0}}(\Lambda_0)\leq\cA^{H}(\Lambda)\leq b$. 
Then the corresponding moduli space  $\mathcal{N}_{\operatorname{hyb}}(\Lambda_0, \Lambda)$ admits uniform $L^\infty$-bounds, which depend only on $a,b$ and $N$.
\end{prop}
\begin{proof}
First observe that by adapting the result in \cite[Prop.\ 6.2]{pasquotto2017} to the hybrid problem we obtain 
\begin{equation}
\sup\left\lbrace \|u^\pm(\pm s)\|_{L^2(S^1)\times\R}\ \Big|\  (u^-,u^+)\in\mathcal{N}_{\operatorname{hyb}}(\Lambda_0, \Lambda),\ s\geq 0\right\rbrace<+\infty.\label{L2bounds}
\end{equation}
Moreover, there exists $\varepsilon_0>0$, such that
\begin{align}
\sup\left\lbrace \begin{array}{c|c}
& (u^-,u^+)\in\mathcal{N}_{\operatorname{hyb}}(\Lambda_0, \Lambda),\ s\leq 0,\\
{\smash{\raisebox{.5\normalbaselineskip}{ $\|u^-(s)\|_{L^\infty\times\R}$}}} &
\|\nabla_{J_0} \mathcal{A}^{H_0}(u^-(s))\|< \varepsilon_0
\end{array}
\right\rbrace<+\infty,\label{EpsiBnds1}\\
\sup\left\lbrace \begin{array}{c|c}
& (u^-,u^+)\in\mathcal{N}_{\operatorname{hyb}}(\Lambda_0, \Lambda),\ s\geq 0,\\
{\smash{\raisebox{.5\normalbaselineskip}{ $\|u^+(s)\|_{L^\infty\times\R}$}}} &
\|\nabla_J \mathcal{A}^{H}(u^+(s))\|< \varepsilon_0
\end{array}
\right\rbrace<+\infty. \label{EpsiBnds2}
\end{align}
and all the bounds depend only on $a,b$ and $N$.

Next, we use a maximum principle argument to prove uniform bounds on the fragment of the Floer trajectories, where the action derivation is greater than $\varepsilon_0$.

For $(u^-,u^+)=((v^-,\eta^-),(v^+,\eta^+))\in\mathcal{N}_{\operatorname{hyb}}(\Lambda_0, \Lambda)$ define a function
\begin{align*}
 r & \hspace{2pt} \colon    (-\infty,0]\times S^1 \to \mathbb{R},\\
r (s,t) &  \coloneqq  \frac{1}{4}\left( \|v^-(s,t)\|^2+\|v^+(-s,t)\|^2\right).
\end{align*}
Using the function $F$ defined as $\frac{1}{4}\|\cdot\|^2$ we can rewrite $r$ as
\[
r(s,t)= F\circ v^-(s,t)+F \circ v^+(-s,t).
\]
The function $F$ is plurisubharmonic, which means that $-dd^\mathbb{C}F=\omega$.

By \eqref{EpsiBnds1} and \eqref{EpsiBnds2} we know that there exists $c_0>0$, such that for every $(s,t)\in r^{-1}([c_0,+\infty))$ we have
\begin{equation}
\|\nabla_{J_0} \cA^{H_0}(u^-(s))\| \ge \varepsilon_0, \qquad  \|\nabla_J \cA^{H}(u^+(-s))\| \geq \varepsilon_0.
\label{NablaBound}
\end{equation}
On the other hand, by Lemma \ref{energy} we have the following bound:
\[
b-a\geq E(u^-,u^+)=\int_{-\infty}^0\|\nabla_{J_0} \cA^{H_0}(u^-(s))\|^2ds+\int_0^{\infty}\|\nabla_J \cA^{H}(u^+(s))\|^2ds.
\]
Therefore, for every connected component $\Omega \subseteq r^{-1}([c_0,+\infty))$ there exist $s_0,s_1 \geq 0$, such that
\[
\Omega \subseteq [s_0,s_1]\times S^1\quad \textrm{and}\quad |s_1-s_0|\leq \frac{b-a}{2\varepsilon_0^2}.
\]
Since \eqref{L2bounds} and \eqref{NablaBound} are both satisfied, we can use \cite[Thm.\ 7.1]{pasquotto2017} and conclude that there exists a function $f \colon  (-\infty,0]\times S^1 \to \mathbb{R}$, such that $\triangle r \geq f$ and the $L^2$-norm of $f$ is bounded on $\Omega$ and the bound depends only on $a,b$ and $N$.

Now, if $\partial \Omega \cap (\{0\}{\times} S^1)=\emptyset$, then we can apply the Aleksandrov Maximum Principle \cite[Thm.\ 9.1]{gilbarg1998} to deduce that there exists $c_1>0$, such that
\begin{equation}
\sup_\Omega r \leq \sup_{\partial \Omega}r+c_1\|f\|_{L^{2}(\Omega)}=c_0+c_1\|f\|_{L^{2}(\Omega)}<+\infty, \label{limit}
\end{equation}
and the bound depends only on $a,b$ and $N$.

However, if $\partial \Omega \cap \big(\{0\}{\times} S^1\big)\neq\emptyset$ then we need to check an additional assumption. More precisely we would like to show that $\partial_sr \geq 0$ on $\partial \Omega \cap (\{0\}{\times} S^1)$ to be able to apply the Aleksandrov Maximum Principle for half cylinders \cite[Thm.\ 2.8]{Abbon2009} for \eqref{limit}.

Denote 
\[
\sigma(v^+)=: (v_0,v_1)\in C^\infty\left(\R\times S^1, T^*\R^k\right)\times C^\infty\left(\R\times S^1,T^*\R^{n-k}\right),
\]
where $\sigma$ is the splitting from \eqref{splitting}. Then (with a slight abuse of notation of $F$)
\begin{equation}\label{dRho}
\partial_sr= dF(\partial_s v^-)-dF(\partial_sv_0)-dF(\partial_sv_1).
\end{equation}
By the coupling condition \eqref{coupling} we have $\eta^-(0)=\eta^+(0)$ and $\partial_tv^-(0,t)=\partial_t v_0(0,t)$. Since $J_0\in \scrJ_\star^k$ and $ J\in \scrJ_\star^n$, we can assume without loss of generality that for $(s,t)\in r^{-1}([c_0,+\infty))$, one has
\[
J_0(u^-(s,t),t) \equiv \mathbb{J}_k, \qquad J(u^+(-s,t),t) \equiv \mathbb{J}_n.
\]
Therefore by \eqref{Floer} we have
\[
\partial_s v_0(0)=\nabla_{\mathbb{J}_k} \cA^{H_0}(v_0(0))=\nabla_{\mathbb{J}_k} \cA^{H_0}(v^-(0))=\partial_s v^-(0).
\]
As a result, the first and second term in \eqref{dRho} cancel each other.

To calculate the last term let us make the following observation:
\begin{align}
dF(\partial_s v_1(0)) & = dF\left(\nabla \cA^{H_1}(v_1(0))\right)\nonumber \\
& = - d^{\mathbb{C}} F\left(\partial_t v_1(0)-\eta^+(0)X_{H_1}(v_1(0))\right). \label{dFC}
\end{align}
By \eqref{coupling} we have that $v_1(0)$ lies in the $0$-section $\chi_0$ and 
$$
d^{\mathbb{C}}F\big|_{T\chi_0}=\frac{1}{2}\left(qdp -pdq\right)\big|_{T\chi_0}=0,
$$
thus the first term in \eqref{dFC} vanishes. For the second term of \eqref{dFC} observe that
\[
d^\mathbb{C}F(X_{H_1})=\langle \nabla H_1, \nabla F\rangle = H_1.
\]
As $H_1$ also vanishes on $\chi_0$, the whole third term in \eqref{dRho} vanishes and we conclude that $\partial_sr=0$ on $\{0\}\times S^1$. This allows us to apply the Aleksandrov Maximum Principle for half cylinders and prove \eqref{limit} also in case where $\partial \Omega \cap (\{0\}\times S^1)\neq\emptyset$.
\end{proof}
\subsection{The Index Computation}\label{ssec:index}
Fix two connected components of the respective critical sets $\Lambda_0 \subseteq \Crit (\mathcal{A}^{H_0})$ and $\Lambda \subseteq \operatorname{Crit}( \mathcal{A}^H)$. In this section we establish that the hybrid moduli problem is Fredholm, and compute the virtual dimension of the the moduli space $\mathcal{M}_{\operatorname{hyb}}(\Lambda_0 , \Lambda)$.

Let us introduce an anti-symplectic involution on $(T^*\R^k,\omega_k)$ by
\begin{equation}
\rho \colon  T^*\R^k \to T^*\R^k, \qquad (q,p)  \mapsto (q,-p). \label{rho}
\end{equation}
The standard symplectic structure $\mathbb{J}_k$ satisfies
\begin{equation}\label{JnRho}
- D\rho \circ \mathbb{J}_k \circ D\rho =\mathbb{J}_k.
\end{equation}
Therefore, we have the associated diffeomorphism:
\[
\mathcal{J}_\star^k \ni J \mapsto -D\rho \circ J \circ D\rho \in \mathcal{J}_\star^k.
\]
Finally, observe that $H_0\circ \rho = H_0$, which gives us
\begin{equation}
\label{XH0Rho}
X_{H_0}(x)= - D\rho[X_{H_0}(\rho (x))].
\end{equation}
As a result we obtain an automorphism of $\operatorname{Crit} (\mathcal{A}^{H_0})$ defined by:
\begin{equation}\label{Lambda0map}
\operatorname{Crit} (\mathcal{A}^{H_0}) \ni (v,\eta) \longmapsto (\rho \circ v, -\eta)\in \operatorname{Crit} (\mathcal{A}^{H_0}).
\end{equation}
Using the notation from \eqref{LambdaEta}, this automorphism maps $\Lambda_0^\eta$ to $\Lambda_0^{-\eta}$.
\begin{theorem}
\label{thm-index}
The hybrid moduli problem is Fredholm, and the space $ \mathcal{M}_{\operatorname{hyb}}(\Lambda_{0} , \Lambda)$ has virtual dimension 
\[
\operatorname{vir\,dim} \mathcal{M}_{\operatorname{hyb}}(\Lambda_{0} , \Lambda) =  \tCZ(\Lambda) - \tCZ(\Lambda_0) + \frac{1}{2} \left( \dim \Lambda_0 + \dim \Lambda\right).
\]
\end{theorem}
\begin{proof}
We prove the theorem in five steps.\\

\paragraph{\textbf{1.}}
Denote by $\mathfrak{y}\coloneqq  \cA^{H_0}(\Lambda_0) \in \operatorname{CritVal}(\cA^{H_0})$, so that by Corollary \ref{cor:CritVal} we have $\Lambda_0=\Lambda_0^{\mathfrak{y}}$. In this first step, we show that $\mathcal{M}_{\operatorname{hyb}}(\Lambda_0^{\mathfrak{y}} , \Lambda)$ is a Fredholm problem. We begin by identifying $\mathcal{M}_{\operatorname{hyb}}(\Lambda_0^{\mathfrak{y}} , \Lambda)$ with a closely related set\footnote{The fact that ``hyb'' is occurs both as a subscript and a superscript is not a typo! They represent formally different spaces.} $\mathcal{M}^{\operatorname{hyb}}(\Lambda_0^{-\mathfrak{y}} , \Lambda)$ defined as follows:

Let $L \subseteq T^*\R^k\times T^*\R^n= \R^{2(n+k)}$ denote the set of elements of the form 
\[
L = \left\{ (a,-b,a,b,c,0_{n-k})\ \Big|\ a,b \in \R^k, c \in \mathbb{R}^{n-k} \right\}.
\]
In other words, $L$ is the preimage of $\Delta_{T^*\R^k}\times \chi_0$ under $\rho \oplus \sigma^{-1}$, where $\rho$ is the anti-symplectic map defined in \eqref{rho},  $ \sigma$ is the symplectic splitting defined in \eqref{splitting}, $\Delta_{T^*\R^k}$ stands for the diagonal in $(T^*\R^k)^2$ and $\chi_0$ is the $0$-section of $T^*\R^{n-k}$. Then $L$ is a Lagrangian submanifold of $(T^*\R^k\times T^*\R^n,\omega_k\oplus \omega_n)$.

For $J_0\in \mathcal{J}_\star^k, J \in \mathcal{J}_\star^n$ we define, using \eqref{JnRho}, an element $\widehat{J}\in \mathcal{J}_\star^{n+k}$ by
\begin{equation}\label{J1}
\widehat{J} \coloneqq \begin{pmatrix}
-D\rho \circ J_0 \circ D\rho & \\
& \;J
\end{pmatrix}.
\end{equation}
An element of $\mathcal{M}^{\operatorname{hyb}}(\Lambda_0^{-\mathfrak{y}} , \Lambda)$ is a pair $(w,\xi)$
with
\begin{equation}\label{eq-w-first}
    w \colon [0,+\infty) \times S^1 \to T^*\R^k\times T^*\R^n, \qquad \xi \colon [0,+\infty) \to \R^2,
\end{equation}
satisfying the equations
\begin{align}
 \partial_s w + \widehat{J}\partial_t w - \xi \cdot X_1 (w)& = 0, \label{eq-w-third}\\
 \partial_s \xi -X_2(w)& = 0, \label{eq-xi-1st}
\end{align}
where
\begin{align}
  (\xi_0,\xi_1) \cdot X_1(w_0,w_1) & \coloneqq  \widehat{J}\begin{pmatrix} \xi_0 X_{H_0}(w_0) \\ \xi_1 X_H(w_1) \end{pmatrix}, \label{eq-w-fourth}\\
  X_2(w) & \coloneqq  \begin{pmatrix} -\int_{S^1}H_0(w_0) \\ -\int_{S^1} H (w_1)\end{pmatrix}, \label{eq-xi-2nd}
\end{align}
together with the limit conditions
\begin{equation}
\lim_{s\to + \infty}w(s)\in \Lambda_0^{-\mathfrak{y}}\times \Lambda, \label{limit2}
\end{equation}
and the coupling conditions
\begin{equation}
w(0,t) \in L, \qquad \forall \,t \in S^1,\qquad \textrm{and} \qquad \xi_1(0)=-\xi_0(0).\label{eq-w-second}
\end{equation}
There is a natural identification of elements $((v_0,\eta_0),(v,\eta))\in\mathcal{M}_{\operatorname{hyb}}(\Lambda_0^{\mathfrak{y}} , \Lambda)$ with elements $(w,\xi)\in\mathcal{M}^{\operatorname{hyb}}(\Lambda_0^{-\mathfrak{y}} , \Lambda)$ given by
\begin{equation}\label{bijection}
w(s,t)\coloneqq  \begin{pmatrix} \rho \circ v_0(-s,t) \\ v(s,t) \end{pmatrix}, \qquad
\xi(s)\coloneqq  \begin{pmatrix} -\eta_0(-s)\\ \eta(s) \end{pmatrix}.
\end{equation}
Indeed, under this identification conditions \eqref{limit1} and \eqref{coupling} become conditions \eqref{limit2} and \eqref{eq-w-second}. 
Furthermore, we compute that 
\[
    \partial_s w = \begin{pmatrix} - D \rho [\partial_s v_0] \\ \partial_s v \\ \end{pmatrix}, \qquad
    \partial_t w  = \begin{pmatrix} D\rho [\partial_t v_0 ]\\ \partial_t v  \\  \end{pmatrix}, \qquad
    \partial_s \xi = \begin{pmatrix} \partial_s \eta_0 \\ \partial_s \eta  \\  \end{pmatrix}.
\]
so that by \eqref{JnRho}, \eqref{XH0Rho}, \eqref{J1}, and \eqref{bijection} we have
\begin{align*}
 \partial_s w + \widehat{J}\partial_t w & = \begin{pmatrix} -D\rho (\partial_s v_0+J_0 \partial_t v_0) \\ \partial_s v+ J \partial_t v \end{pmatrix} =\begin{pmatrix} -\eta_0 D\rho \circ J_0 X_{H_0}(v_0) \\ \eta J X_H (v) \end{pmatrix}\\
& = \begin{pmatrix} \eta_0 D\rho \circ J_0 \circ D\rho [X_{H_0}(w_0)] \\ \eta J X_H (w_1) \end{pmatrix} = \widehat{J}\begin{pmatrix} \xi_0 X_{H_0}(w_0) \\ \xi_1 X_H(w_1) \end{pmatrix}.
\end{align*}
In this way the Rabinowitz Floer equations \eqref{Floer} become equivalent to \eqref{eq-w-third}--\eqref{eq-xi-2nd}.  So $ \mathcal{M}_{\operatorname{hyb}}(\Lambda_0^{\mathfrak{y}} , \Lambda)$ is a Fredholm problem if and only if $\mathcal{M}^{\operatorname{hyb}}(\Lambda_0^{-\mathfrak{y}} , \Lambda)$ is a Fredholm problem.

Fix $r >2$. Let $\mathcal{B}(\Lambda_0^{-\mathfrak{y}},\Lambda)$ denote the Banach manifold of pairs $(w,\xi)$ as in \eqref{eq-w-first} satisfying \eqref{limit2} and \eqref{eq-w-second}. For $(w,\xi)\in \mathcal{B}(\Lambda_0^{-\mathfrak{y}},\Lambda)$, we do not require equations \eqref{eq-w-third}--\eqref{eq-xi-2nd} to hold. Instead, $(w,\xi)$ should be locally of class $W^{1,r}$ and converge for $s\rightarrow +\infty$ exponentially to an element $(x_0,x_1) \in \Lambda_0^{-\mathfrak{y}}\times\Lambda$. The tangent space of $\mathcal{B}(\Lambda_0^{-\mathfrak{y}},\Lambda)$ at $(w,\xi)$ can be identified with
\begin{align*}
T_{(w,\xi)} \mathcal{B}(\Lambda_0^{-\mathfrak{y}},\Lambda)&\cong \mathcal{W}^{1,r}_{\delta,L}\oplus \mathcal{W}^{1,r}_{\delta,\Delta}\oplus  T_{ x_0}\Lambda_0\oplus T_{x_1}\Lambda,\\
\text{where }\qquad\qquad\mathcal{W}^{1,r}_{\delta,L} & \coloneqq \left\lbrace z \in W^{1,r}_{\delta} \left([0, \infty) \times S^1 , \mathbb{R}^{2k} \times \mathbb{R}^{2n} \right) \ \Big|\ z(0,t) \in L \right\rbrace,\\
\mathcal{W}^{1,r}_{\delta,\Delta} & \coloneqq \left\lbrace \zeta \in W^{1,r}_{\delta}\left([0,\infty), \mathbb{R}^2\right) \ \Big|\ \zeta_1(0)=-\zeta_0(0) \right\rbrace.
\end{align*}
Here, $\delta$ indicates that we are working with weighted Sobolev spaces with weight $\gamma(s)=e^{\delta s}$ (as our asymptotic operators will not be bijective).
 Explicitly, $f\in \mathcal{W}^{1,r}_\delta$ if and only if $f\cdot \gamma\in \mathcal{W}^{1,r}$. Let $\mathcal{E}$ be the Banach bundle over $\mathcal{B}(\Lambda_0^{-\mathfrak{y}},\Lambda)$ whose fiber at $(w,\xi)$ is given by
 \[\mathcal{E}_{(w,\xi)}\coloneqq L^r_{\delta}\left( [0,\infty) \times S^1, \mathbb{R}^{2k} \times \mathbb{R}^{2n} \right) \times L^r_{ \delta}\left([0,\infty), \mathbb{R}^2\right).\]
 Define a section
 \begin{equation}
 \bar{\partial} \colon \mathcal{B}(\Lambda_0^{-\mathfrak{y}},\Lambda)\to \mathcal{E}, \label{eq:Banach-sec}
 \end{equation}
 by the left-hand side of the equations \eqref{eq-w-third} and \eqref{eq-xi-1st}. Then $\mathcal{M}^{\operatorname{hyb}}(\Lambda_0^{-\mathfrak{y}} ; \Lambda)=\bar{\partial}^{-1}(0)$. Since \eqref{eq-w-second} is a Lagrangian boundary condition and \eqref{eq-w-third} is a perturbed Cauchy-Riemann equation on a half-cylinder, coupled with a pair of ordinary differential equations \eqref{eq-xi-1st} for $\xi$, the proof that the linearisation 
 \[
 D\bar{\partial} \colon T_{(w,\xi)}\mathcal{B}(\Lambda_0^{-\mathfrak{y}},\Lambda)\to \mathcal{E}_{(w,\xi)},
 \]
  is a Fredholm operator is a routine argument (see for instance \cite[Sec. 5.4]{AbbondandoloSchwarz2010}), and we will not go over the details here. \\

\paragraph{\textbf{2.}}
Computing the index however is somewhat less standard, for two reasons:
\begin{enumerate}[label=(\roman*)]
    \item In contrast to standard index computations in Rabinowitz Floer homology, here there are two Lagrange multipliers to worry about.
    \item We are on a half cylinder, and thus arguments using the spectral flow are not directly applicable.
\end{enumerate}
This computation is very similar to \cite[Theorem 4.12]{AbbonMerry2014}. In order to keep the exposition reasonably clean, we perform the calculation only in the case where both $ \mathcal{A}^{H_0}(\Lambda_0^{-\mathfrak{y}})$ and $ \mathcal{A}^H(\Lambda)$ are non-zero. The other cases are easier and are left to the reader.

In this step, we formulate an appropriate local model for the problem $ \mathcal{M}^{\operatorname{hyb}}(\Lambda_0^{-\mathfrak{y}} , \Lambda)$ by linearising equations \eqref{eq-w-third}--\eqref{eq-xi-2nd} at a solution $(w,\xi) \in \mathcal{M}^{\operatorname{hyb}}(\Lambda_0^{-\mathfrak{y}} , \Lambda)$, and express the virtual dimension of  $\mathcal{M}^{\operatorname{hyb}}(\Lambda_0^{-\mathfrak{y}} , \Lambda)$ in terms of the index of this linearised operator. 

Write $m_0\coloneqq  k, m_1\coloneqq n$ and let 
\[
    P_i \in W^{1, \infty}\left( [0,+\infty) \times S^1 , \operatorname{Mat}(2m_i)\right),
\]
for $i=0,1$ denote two paths of matrices that extend to the compactification $[0, +\infty] \times S^1$ in such a way that for each $i=0,1$
\[
\lim_{s_0 \rightarrow + \infty}\underset{(s,t)\in[s_0,\infty )\times S^1} {\operatorname{ess\, sup}}\Big(\left|\partial_s P_i(s,t)\right|+\left|\partial_t P_i(s,t)-\partial_t P_i(t,\infty)\right|\Big)=0.
\]
Assume in addition that the limit matrices $S_i(t) \coloneqq P_i(+ \infty, t)$ are symmetric:
\begin{equation}
    \label{eq-limit-S}
    S_0(t) \in \operatorname{Sym}( \mathbb{R}^{2k}), \qquad S_1(t) \in \operatorname{Sym}( \mathbb{R}^{2n}).
\end{equation}
Next, for $i=0,1$ let 
\[
\beta_i \in W^{1, \infty}( [0,\infty) \times S^1 , \mathbb{R}^{2m_i}),\\
\]
denote two vector-valued paths which extend to the compactification $[0,+\infty] \times S^1$ in such a way that for each $i = 0,1$,
 \[
\lim_{s_0 \rightarrow + \infty}\underset{(s,t)\in (s_0,\infty)\times S^1} {\operatorname{ess\,sup}}\Big(\left|\partial_{s}\beta_i(s,t)\right|+\left|\partial_{t}\beta_i(s,t)-\partial_{t}\beta_i(t,+\infty)\right|\Big)=0.
\]
Abbreviate $b_i(t) \coloneqq \beta_i(+\infty,t)$, and assume that 
\[
b_i  \in W^{1,2}(S^1, \mathbb{R}^{2m_i}) \cap \operatorname{range}(\mathbb{J}_{m_i}\partial_t +S_i), \qquad i = 0,1.
\]
We are now ready to introduce an appropriate local model. Consider the operator $\mathbf{D}  \colon \mathcal{W}^{1,r}_{ \delta, L} \times \mathcal{W}^{1,r}_{\delta,\Delta}   \to \mathcal{E}$ defined by\footnote{We somewhat ambiguously switch between row and column notation, favouring whichever is cleaner for any given equation.}
\begin{equation}\label{eq-D}
    \mathbf{D} \begin{pmatrix}
z_0 \\ z_1 \\ \zeta_0 \\ \zeta_1 
\end{pmatrix}
\coloneqq 
\begin{pmatrix}
\partial_s z_0 + \mathbb{J}_k \partial_t z_0 + P_0 z_0 + \zeta_0 \beta_0 \\
\partial_s z_1 + \mathbb{J}_n \partial_t z_1 + P_1 z_1 + \zeta_1 \beta_1 \\
\zeta_0' + \int_{S^1} \langle z_0, \beta_0  \rangle \,dt\\
\zeta_1' + \int_{S^1} \langle z_1, \beta_1  \rangle \,dt
\end{pmatrix}
\end{equation}
Note that $\mathbf{D}$ is the restriction to $\mathcal{W}^{1,r}_{\delta,L}\oplus\mathcal{W}^{1,r}_{\delta,\Delta}$ of the linearisation $D\bar{\partial}$ of the problem \eqref{eq:Banach-sec} at a solution $(w, \xi)\in\mathcal{M}^{\operatorname{hyb}}(\Lambda_0^{-\mathfrak{y}} ; \Lambda)$ of \eqref{eq-w-third}--\eqref{eq-xi-2nd}, when viewed in a suitable symplectic trivialisation.
Explicitly, for $H=H_0+H_1$ satisfying \eqref{defH} we have
\begin{align}
S_0 & =  -\eta_0 \mathbb{J}_k D_{v_0}X_{H_0}, && b_0 (t)  = -\mathbb{J}_k [ X_{H_0}] (v_0(t)),\label{Sb0}\\
S_1 & = -\eta \mathbb{J}_n D_v X_H, && b_1(t)  = - \mathbb{J}_n [X_H] (v(t)),\label{Sb1}
\end{align}
for some $(v_0,\eta_0)\in \Lambda_0^{-\mathfrak{y}}, (v,\eta)\in \Lambda$.
Consequently, we have 
\begin{equation}\label{dimM}
 \operatorname{vir\,dim} \mathcal{M}^{\operatorname{hyb}}(\Lambda_0^{-\mathfrak{y}} , \Lambda)  =  \operatorname{ind} D\bar{\partial}= \operatorname{ind} \mathbf{D} + \dim \Lambda_0 + \dim \Lambda.
\end{equation}
$\ $
\paragraph{\textbf{3.}}
Our aim is to homotope the operator $\mathbf{D}$ from \eqref{eq-D} through Fredholm operators into a new operator of product form 
\begin{equation}\label{Dnew}
\mathbf{D}^{\operatorname{new}} (z, \zeta) = \left( \mathbf{D}_1(z), \mathbf{D}_2( \zeta) \right).
\end{equation}
The operator $\mathbf{D}^{\operatorname{new}}$ decouples the perturbed Cauchy-Riemann equation satisfied by $z$ and the ordinary differential equation satisfied by $ \zeta$. This will allow us to compute the index of $\mathbf{D}$:
\begin{align}
    \operatorname{ind} \mathbf{D}& = \operatorname{ind} \mathbf{D}^{\operatorname{new}}  & &\text{since the index is constant along deformations,} \nonumber \\
  &  = \operatorname{ind} \mathbf{D}_1 + \operatorname{ind}\mathbf{D}_2 & & \text{by additivity.} \label{eq-index-sum}
\end{align}
The operator $\mathbf{D}^{\operatorname{new}}$ is constructed using methods from \cite[App.\ C]{CieliebakFrauenfelder2009}.
The index of $ \mathbf{D}_1$ can be computed using the methods from \cite{AbbondandoloSchwarz2010}. Meanwhile the index of $ \mathbf{D}_2$ can be computed by hand.

In order to construct the homotopy we first need to compute the perturbation term defined in \cite[Def.\ C.3]{CieliebakFrauenfelder2009} as follows: for $i=0,1$ and $\delta>0$ choose $z_i \in W^{1,2}(S^1, \mathbb{R}^{2m_i})$ such that 
\begin{equation}\label{Av=b}
(\mathbb{J}_{m_i}\partial_t +S_i+\delta \operatorname{Id})z_i = b_i,
\end{equation}
and define
\begin{equation}
\label{tau1}
\tau_i \coloneqq  \int_{S^1} \langle z_i(t), b_i(t) \rangle \,dt.
\end{equation}
Since the operator $\mathbb{J}_{m_i}\partial_t +S_i$ is self-adjoint, the number $\tau_i$ does not depend on the choice of $z_i$. We will show that in our setting, for $\delta>0$ small enough
\begin{equation}\label{tau2}
\tau_0, \tau_1>0.
\end{equation}
First observe, that by Lemma \ref{lem:orbits} every periodic orbit $(v,\eta)\in \Crit(\cA^H), \eta\neq 0$ is of the form $( \sigma^{-1}(v_1,0),\eta)$ with $(v_1,\eta)\in \Crit(\cA^{H_0})$. As a result 
\[
dH(v)=dH \circ \sigma^{-1} (v_1,0)=dH_0(v_1).
\]
Therefore, if $(z_0,z)$ is a solution to \eqref{Av=b} and
 $ (z_1,z_2) \coloneqq \sigma(z) \in W^{1,2}(S^1,T^*\R^k)\times  W^{1,2}(S^1,T^*\R^{n-k})$, then by \eqref{Sb1} we get
\[
\tau_1 = \int_{S^1}d_v H (z)\,dt =\int_{S^1}d_{v_1}H_0 (z_1)\,dt.
\]
Therefore, without loss of generality we can assume that $z_2=0$. Consequently, by \eqref{CritEta}, \eqref{Sb0} and \eqref{Sb1}, equation \eqref{Av=b} translates to
\begin{equation}\label{Av=b2}
(\mathbb{J}_k\partial_t +\eta_i A_0 +\delta \operatorname{Id})z_i = A_0 \exp[t \eta_i \mathbb{J}_k A_0] v_i(0),
\end{equation}
where $(v_0,\eta_0)\in \Lambda_0^{-\mathfrak{y}}, ( \sigma^{-1}(v_1,0),\eta_1)\in \Lambda \subseteq \Crit (\cA^H)$.

It is easy to check that
\[
z_i(t)\coloneqq \frac{1}{\delta} A_0 \exp[t \eta_i \mathbb{J}_k A_0] v_i(0),
\]
solves \eqref{Av=b2} and as a result for $i=0,1$ we have
\[
\tau_i = \frac{1}{\delta}\int_{S^1} \left\|A_0  \exp[t \eta_i \mathbb{J}_k A_0] v_i(0)\right\|^2dt>0,
\]
which proves \eqref{tau2}.\\

\paragraph{\textbf{4.}}
In this step we will construct the homotopy between the operator $\mathbf{D}$ from \eqref{eq-D} and a new operator $\mathbf{D}^{\operatorname{new}}$ of the form as in \eqref{Dnew}.

If we set 
\[
Q = \begin{pmatrix}
\mathbb{J}_k \partial_t + P_0  &  \\
& \mathbb{J}_n \partial_t + P_1
\end{pmatrix}, \qquad  \beta = \begin{pmatrix}
\beta_0 \\
\beta_1
\end{pmatrix}, 
\qquad
B \begin{pmatrix}
z_0 \\ z_1 
\end{pmatrix}
= 
\begin{pmatrix}
\int_{S^1} \langle z_0, \beta_0  \rangle \,dt\\ 
\int_{S^1} \langle z_1, \beta_1  \rangle \,dt
\end{pmatrix},
\]
then the operator $ \mathbf{D}$ can be written as 
\[
\mathbf{D} = \partial_s + \begin{pmatrix}
Q & \beta \\
B & 0
\end{pmatrix}.
\]
We now consider the homotopy $\{ \mathbf{D}^{\theta }\}_{ \theta \in [0,1]}$ of operators given by
\[
\mathbf{D}^{\theta} = \partial_s +  \begin{pmatrix}
Q & (1 - \theta)  \beta \\
(1 - \theta) B & c( \theta)
\end{pmatrix},
\]
where $c  \colon  [0, 1] \to \mathbb{R}^2$ is defined by
\begin{equation}
\label{eq-c}
c( \theta)  \coloneqq 
\begin{pmatrix}
\theta \tau_0 \\
\theta \tau_1
\end{pmatrix}.
\end{equation}
By \cite[Thm. C.5]{CieliebakFrauenfelder2009} the operators $\mathbf{D}^{\theta}$ are all Fredholm of the same index. We define
\begin{align*}
 \mathbf{D}^{\operatorname{new}} & \coloneqq \mathbf{D}^1=(\mathbf{D}_1,\mathbf{D}_2), & & \textrm{where} \\
 \mathbf{D}_1 & \colon \mathcal{W}^{1,r}_{\delta, L}  \to L^r_{\delta} \left( [0,\infty) \times S^1, \mathbb{R}^{2n+2k}\right),
  & z & \mapsto   \partial_s z + Q z,\\
\mathbf{D}_2 & \colon \mathcal{W}^{1,r}_{\delta, \Delta}  \to L^r_{ \delta} \left( [0,\infty), \mathbb{R}^2\right), &
 \begin{pmatrix}
 \zeta_0 \\ \zeta_1
 \end{pmatrix}  &\mapsto \begin{pmatrix}
 \zeta_0' + \tau_0 \zeta_0 \\ \zeta_1' + \tau_1\zeta_1 
 \end{pmatrix} .
\end{align*}

\paragraph{\textbf{5.}}
In this final step we compute the index of $ \mathbf{D}_1$ and $ \mathbf{D}_2$. For this denote by $ \Psi_i$ for $i=0,1$ the following paths of symplectic matrices:
\begin{align*}
    \Psi_0 & \colon [0,1] \to \operatorname{Symp}( \mathbb{R}^{2k}, \omega_k), && \begin{cases}
    \Psi'_0(t) = \mathbb{J}_k S_0(t) \Psi_0(t),\\
    \Psi_0(0) = \operatorname{Id}_{2k},
    \end{cases} \\
        \Psi_1 & \colon [0,1] \to \operatorname{Symp}( \mathbb{R}^{2n}, \omega_n), && \begin{cases}
    \Psi'_1(t) = \mathbb{J}_n S_1(t) \Psi_1(t),\\
    \Psi_1(0) = \operatorname{Id}_{2n},
    \end{cases}
\end{align*}
Denote by $\CZ(\Psi_i) \in \frac{1}{2} \mathbb{Z}$ the (full) Conley-Zehnder index of these paths, and let $ \operatorname{null}( \Psi_i) \coloneqq \dim \left( \ker \Psi_i(1) - \operatorname{Id} \right) $ denote the nullity. Then the Fredholm index of $\mathbf{D}_1$ is given by \cite[Theorem 5.25]{AbbondandoloSchwarz2010} as
\begin{equation}
\label{eq-index-D1}
\operatorname{ind} \mathbf{D}_1 = \sum_{i=0}^1 \left( \CZ(\Psi_i) - \frac{1}{2}\operatorname{null}(\Psi_i) \right)+m(L),
\end{equation}
where $m(L)$ stands for the correction term coming from the boundary conditions.
Actually \eqref{eq-index-D1} is much simpler than the general statement in \cite[Thm. 5.25]{AbbondandoloSchwarz2010}, since we only have a single boundary condition (rather than a set of jumping boundary conditions). The correction term $m(L)$ is computed by \cite[Thm. 5.25]{AbbondandoloSchwarz2010} to be
\[
m(L)\coloneqq \frac{\dim T^*\R^{n+k}}{2} -\frac{1}{2}\dim \Delta_{T^*\R^{n+k}} - \dim L\\
 + \dim \left(\Delta_{T^*\R^{n+k}}\cap(L\times L) \right),
\]
where $\Delta_{T^*\R^{n+k}}$ is the diagonal in $(T^*\R^{n+k})^2$. In this case,
\begin{align*}
\dim \Delta_{T^*\R^{n+k}} & =2(n+k), \\
\dim \left(\Delta_{T^*\R^{n+k}}\cap(L\times L) \right) & =\dim L=n+k,
\end{align*}
hence\footnote{Sanity check: The same computation works for any Lagrangian $L$. Thus the correction term is always $0$ when there is but a single Lagrangian boundary condition. This can also be proved directly, without appealing to the machinery of \cite{AbbondandoloSchwarz2010}.} $m(L)=0$. On the other hand, by Proposition \ref{prop:muCZ} we have
\begin{align*}
\CZ(\Psi_0) & =\CZ(\Lambda_0^{-\mathfrak{y}})= - \CZ(\Lambda_0^\mathfrak{y})=-\tCZ(\Lambda_0),\\
\CZ(\Psi_1) & = \CZ(\Lambda)= \tCZ(\Lambda).
\end{align*}
Moreover, by Lemma \ref{lem:orbits} and \eqref{CritEta} we have
\begin{align*}
\dim \Lambda_0 & = \dim \Lambda_0^{-\mathfrak{y}}= \dim \left( \ker \left(\Psi_0(1) - \operatorname{Id}\right) \cap \Sigma_0\right)  =\operatorname{null}( \Psi_0)-1,\\
\dim \Lambda & = \dim \left(\ker \left( \Psi_1(1) - \operatorname{Id}\right) \cap \left(\Sigma_0\times \{0\}\right)\right) =\operatorname{null}( \Psi_1)-1,
\end{align*}
which gives
\begin{equation}\label{D1}
\operatorname{ind} \mathbf{D}_1 =   \tCZ(\Lambda) - \tCZ(\Lambda_0) -\frac{1}{2}\left( \dim \Lambda_0 + \dim \Lambda \right)-1.
\end{equation}
Now we would like to compute $\operatorname{ind}  \mathbf{D}_2$. Note that
\[
\operatorname{ind} \mathbf{D}_2 = 2 \operatorname{ind} \widetilde{\mathbf{D}}_2- 1,
\]
where the $-1$ comes from the boundary condition $\zeta_1(0)=-\zeta_0(0)$ and $\widetilde{\mathbf{D}}_2$ is an operator of the form
\begin{gather*}
\widetilde{\mathbf{D}}_2 \colon  W^{1,r}_{\delta}\left([0,+\infty), \R\right)  \to  L^r_{ \delta}\left([0,+\infty), \R\right),\\
f \mapsto f'+\tau f.
\end{gather*}
Note that $\dim \operatorname{coker} \widetilde{\mathbf{D}}_2 =0$. However, since  we are working with weighted Sobolev spaces on a positive half-cylinder, the dimension of $\ker \widetilde{\mathbf{D}}_2=1$, since in both cases $\tau>0$ by \eqref{tau1}. Consequently, $\operatorname{ind} \mathbf{D}_2 = 1$. Combining this with \eqref{dimM} and \eqref{D1}
we have
\begin{align*}
   \operatorname{vir\,dim} \mathcal{M}_{\operatorname{hyb}}(\Lambda_0 , \Lambda) & = \operatorname{ind} \mathbf{D}^{\operatorname{new}}  + \dim \Lambda_0 + \dim \Lambda \\
   & = \tCZ(\Lambda) - \tCZ(\Lambda_0) + \frac{1}{2} \left( \dim \Lambda_0 + \dim \Lambda\right).
\end{align*}
This completes the proof of Theorem \ref{thm-index}.
\end{proof}
\subsection{Automatic transversality.}\label{sec:trans}
If the space $\mathcal{M}_{\operatorname{hyb}}(\Lambda_0 ,\Lambda)$ contains no stationary solutions, transversality can be achieved for generic choices of almost complex structures. This is a standard -- albeit, difficult -- argument, which involves no new ideas not already present in the proof that the moduli spaces defining the Rabinowitz Floer complex are generically transverse (see for instance \cite[Thm.\ 4.11]{AbbonMerry2014} or \cite[Thm.\ 2]{Wisniewska2017} for the non-compact case). We therefore omit the proof.

The case where  $\mathcal{M}_{\operatorname{hyb}}(\Lambda_0 ,\Lambda)$ admits stationary solutions is somewhat less standard, however. In the presence of stationary solutions it is not possible to obtain regularity by perturbing $J$, and one must therefore prove transversality ``by hand''. This is the content of the following section.

Fix $\eta\in \operatorname{CritVal}(\cA^{H_0})=\operatorname{CritVal}(\cA^H)$ and let $\Lambda_0^\eta \subseteq \Crit(\cA^{H_0})$ and $\Lambda^\eta \subseteq \Crit(\cA^H)$ be the corresponding connected components of the respective critical sets. Then  $\mathcal{M}_{\operatorname{hyb}}( \Lambda_0^\eta , \Lambda^\eta)$ consists entirely of stationary solutions by Lemma \ref{lem:stationary} and

\begin{equation}\label{eq:LambdaLambda}
\begin{aligned}
\text{for }\eta\neq 0 \colon  \qquad \mathcal{M}_{\operatorname{hyb}}(\Lambda_0^\eta,\Lambda^\eta)&\cong \Lambda_0^\eta\cong \Lambda^\eta,\\
\text{for }\eta= 0 \colon  \qquad \mathcal{M}_{\operatorname{hyb}}(\Lambda_0^\eta,\Lambda^\eta)&\cong \Sigma_0\times\chi_0,
\end{aligned}
\end{equation}
where $\chi_0\subset T^\ast\R^{n-k}$ denotes the 0-section.

\begin{prop}\label{prop:trans}~
\begin{enumerate}[label=\emph{(\roman*)}]
\item 
\label{at2}
The problem $ \mathcal{M}_{\operatorname{hyb}}( \Lambda_0^\eta,\Lambda^\eta )$ is always transversly cut out.
\item 
\label{at1}
 The kernel of the linearisation $D\bar{\partial}$ of the problem $ \mathcal{M}^{\operatorname{hyb}}( \Lambda_0^{-\eta},\Lambda^\eta)$ (see \eqref{eq:Banach-sec}) at a stationary solution $(w,\xi)\in\mathcal{M}^{\operatorname{hyb}}( \Lambda_0^{-\eta},\Lambda^\eta)$ has dimension $\frac{1}{2}(\dim\Lambda+\dim\Lambda_0)$. 
\end{enumerate}
\end{prop}

\begin{proof}[Proof of Proposition \ref{prop:trans}]
 The proof presented below is an adjustment of \cite[Lem. 4.14]{AbbonMerry2014} to our setting. Note that part \emph{\ref{at2}} is an immediate corollary of part \emph{\ref{at1}}. Indeed, by the correspondence \eqref{bijection}, we have that $\mathcal{M}_{\operatorname{hyb}}( \Lambda_0^\eta,\Lambda^\eta)$ is transversely cut out if and only if $\mathcal{M}^{\operatorname{hyb}}( \Lambda_0^{-\eta},\Lambda^\eta)$ is transversely cut out. The latter holds if and only if the operator $D\bar{\partial}$ is surjective at $(w,\xi)$. Now, $D\bar{\partial}$ is a Fredholm operator of index $\frac{1}{2}(\dim\Lambda+\dim\Lambda_0)$ by Theorem \ref{thm-index} (see also \eqref{dimM}), as $\tCZ(\Lambda_0)=\tCZ(\Lambda)$ by Proposition \ref{prop:muCZ}. Part \emph{\ref{at1}} tells us that $ \dim \ker D\bar{\partial} = \frac{1}{2}(\dim\Lambda+\dim\Lambda_0)$. Thus $D\bar{\partial}$ is surjective, as required.

Fix $(w,\xi) \in \mathcal{M}^{\operatorname{hyb}}( \Lambda_0^{-\eta},\Lambda^\eta)$ and let $(z, \zeta)\in \ker\big(D_{(w,\xi)}\bar{\partial}\big)$ be arbitrary, i.e.\ $(z,\zeta)$ is a solution to the linearised equation $D_{(w,\xi)}\bar{\partial}(z,\zeta)=0$. In the following, we write $w=(w_0,w_1,w_2)$ and $z=(z_0,z_1,z_2)$ with respect to the splitting $T^\ast\R^k\times T^\ast\R^n=T^\ast\R^k\times T^\ast\R^k\times T^\ast\R^{n-k}$. 
Denote $(x_0,x_1)\coloneqq \lim_{s\to+\infty}(w, \xi)(s)\in \Lambda_0^{-\eta}\times\Lambda^\eta$.
Then $(z,\zeta)\in\mathcal{W}^{1,r}_{ \delta, L} \oplus W^{1,r}_{\delta, \Delta}\oplus T_{x_0}\Lambda_0^{-\eta} \oplus T_{x_1}\Lambda^\eta$ and satisfies explicitly
\begin{equation}
    \label{b}
    \frac{d}{ds}(z_0,\zeta_0) + \nabla^2_{(w_0,\xi_0)} \mathcal{A}^{H_0}(z_0,\zeta_0) = 0, \qquad \frac{d}{ds}(z_1,z_2,\zeta_1) + \nabla^2_{(w_1,w_2,\xi_1)} \mathcal{A}^{H}(z_1,z_2,\zeta_1)  = 0,
\end{equation}
the coupling conditions
\begin{equation}
\label{c}
  \zeta_1(0)=-\zeta_0(0), \qquad \forall\ t \in S^1 \colon  \quad z_0(0,t)= D\rho[ z_1(0,t)],\quad z_2(0,t) \in T\chi_0,
\end{equation}
and the asymptotic conditions 
\begin{equation}
    \label{d}
    \lim_{s \to + \infty } (z, \zeta)(s) =(y_0,y_1)\in T_{x_0}\Lambda_0^{-\eta} \times T_{x_1}\Lambda^\eta.
\end{equation}
To prove part \emph{\ref{at1}}, we show that any solution $(z,\zeta)$ of $D_{(w,\xi)}\bar{\partial}(z,\zeta)=0$ is constant. Then, we have by \eqref{d} that $(z,\zeta)=(y_0,y_1)\in T_{x_0}\Lambda_0^{-\eta} \oplus T_{x_1}\Lambda^\eta$, which implies with the coupling condition \eqref{c} and with \eqref{eq:LambdaLambda} that the space of such $(z,\zeta)$ has dimension $\frac{1}{2}(\dim\Lambda+\dim\Lambda_0)$.

First recall that the second derivative of the functional $\mathcal{A}^{H}$ at a loop $(v,\eta)$, is a symmetric, bilinear operator on $W^{1,2}(S^{1},v^*T^*\R^n)\times \mathbb{R})$ given by:
\begin{equation}\label{d2AH}
 d^2_{(v,\eta)}\mathcal{A}^{H} ((\xi, \sigma),(\xi, \sigma)) = \int \omega(\xi,\partial_{t}\xi) - \eta\int \operatorname{Hess}_{v} H(\xi,\xi) -2\sigma\int dH(\xi).
\end{equation}
We define a function $\varphi \colon [0,+\infty)\to \R$ in the following way:
\[
\varphi(s)\coloneqq \big\|\big(z(s, \cdot),\zeta(s)\big)-(y_0,y_1)\big\|^2_{L^2(S^1)}.
\]
By assumption $\varphi \in L^r_\delta([0,+\infty),\R)$. As $\mathcal{A}^H$ and $\mathcal{A}^{H_0}$ are Morse-Bott (cf.\ \eqref{eq:MorseBott}), we have $(y_0,y_1)\in T_{x_0}\Lambda_0^{-\eta} \times T_{x_1}\Lambda^\eta=\ker \nabla^2_{x_0}\mathcal{A}^{H_0}\oplus\ker\nabla^2_{x_1}\mathcal{A}^H$ and hence\footnote{Observe that $(w_0,\xi_0)\equiv x_0$ and $(w_1,w_2,\xi_1)\equiv x_1$ for all $s\in [0,\infty)$.}
\begin{align*}
\varphi'(s) & = 2 \left( d^2_{(w_0,\xi_0)}\mathcal{A}^{H_0} ((z_0, \zeta_0),(z_0, \zeta_0)) + d^2_{(w_1,w_2,\xi_1)}\mathcal{A}^H ((z_1, z_2, \zeta_1),(z_1, z_2, \zeta_1)) \right), \\
\varphi''(s) & = 4 \left(\left\|\nabla^2_{(w_0,\xi_0)} \cA^{H_0}(z_0, \zeta_0)\right\|^2 + \left\|\nabla^2_{(w_1,w_2,\xi_1)} \cA^H(z_1,z_2, \zeta_1)\right\|^2\right)\geq 0.
\end{align*}
Thus on one hand $\varphi''(s)\geq 0$ and $\varphi$ is convex and on the other hand $\lim_{s\rightarrow\infty}\varphi (s)=0$. That implies that either $\varphi'(0)<0$ or $\varphi'(0)=0$ and $\varphi$ is constant, equal to 0 everywhere. We will show that $\varphi'(0)=0$, which implies that $(z,\zeta)$ is constantly equal to $(y_0,y_1)$. First, we show that
\begin{equation}\label{d2AH1=0}
d^2_{(w_2,\xi_1)}\mathcal{A}^{H_1} (( z_2, \zeta_1),(z_2, \zeta_1))\big|_{s=0}=0.
\end{equation}
Recall from the coupling conditions \eqref{eq-w-second} and \eqref{c} that $w_2(0,t)\in\chi_0$ and $z_2(0,t)\in T\chi_0$ for all $t\in S^1$. Writing \eqref{d2AH1=0} in the form \eqref{d2AH} as 3 integrals, we find that the first integral vanishes, as $\chi_0$ is a Lagrangian submanifold of  $\left(T^*\mathbb{R}^{n-k},\omega_{n-k}\right)$. On the other hand, for every $x\in \chi_0$ we can identify $T_x\chi_0$ with $\mathbb{R}^{n-k}\times\{0\}$ and therefore by \eqref{defH} for every $y\in T_x\chi_0$ we have
\[
\operatorname{Hess}_x H_1(y,y)=y^TA_1y=2 H_1(y)=0,
\]
since $\chi_0 \subseteq H_1^{-1}(0)$ by Remark \ref{H2}.
Consequently, the second integral in \eqref{d2AH} vanishes. Analogously, $T\chi_0\subseteq TH_1^{-1}(0)=\ker(dH_1)$, hence the third integral in \eqref{d2AH} also is $0$. This proves \eqref{d2AH1=0}. Next we will show that
\[
\left( d^2_{(w_0,\xi_0)} \mathcal{A}^{H_0} ((z_0, \zeta_0),(z_0, \zeta_0))  +d^2_{(w_1,\xi_1)}\mathcal{A}^H ((z_1, \zeta_1),(z_1, \zeta_1))\right)\Big|_{s=0}=0.
\]
 Observe that by the coupling conditions \eqref{eq-w-second} and \eqref{c} we have for all $t\in S^1$ that $w_0(0,t) = \rho \circ w_1(0,t)$ and $z_0(0,t) = D\rho[ z_1(0,t)]$ and hence
\[
 d^2_{(w_0,\xi_0)} \mathcal{A}^{H_0} ((z_0, \zeta_0),(z_0, \zeta_0))\Big|_{s=0} = d^2_{(\rho \circ w_1,-\xi_1)}\mathcal{A}^H ((D\rho[z_1], -\zeta_1),(D\rho[z_1], -\zeta_1))\Big|_{s=0}
\]
Thus, the first corresponding integrals in \eqref{d2AH} have opposite signs, since $\rho$ is anti-symplectic. The second integrals in \eqref{d2AH} have opposite signs, since $\xi_0(0)=-\xi_1(0)$ and $H_0 \circ \rho = H_0$ and thus $\operatorname{Hess} H_0(D\rho\, \cdot\,, D\rho\, \cdot\, )= \operatorname{Hess} H_0$. The third integrals  in \eqref{d2AH} have opposite signs, since $\zeta_0(0)=-\zeta_1(0)$ and for every $(x,y)\in T\R^k, \ dH_{\rho (x)}(D\rho[y])=dH_x(y)$. Consequently $\varphi'(0)=0$.
\end{proof}
\section{Computation of the Rabinowitz Floer homology}
\subsection{Building the isomorphism}
The main goal of this  section will be the proof of Theorem \ref{thm:iso}.
Throughout this section we will consider the following setting: let $H \colon T^*\R^n\to \R$ and $H_0 \colon T^*\R^k\to \R$ be Hamiltonians satisfying \eqref{defH} and let $J_0 \in \mathcal{J}_\star^k, J\in \mathcal{J}_\star^n$ be two $2$-parameter families of $\omega_0$-compatible almost complex structures, such that the couples $(H_0,J_0)$ and $(H,J)$ are regular in the sense of Definition \ref{def:regular}.

By Lemma \ref{lem:orbits}, the embedding $T^*\R^k\cong T^*\R^k\times \{0\}^{2(n-k)}\hookrightarrow T^*\R^n$ induces an inclusion\footnote{This is not to be confused with the inclusion $j$ from \eqref{eqj}.} 
\begin{equation}
    \label{eqi}
    i \colon \operatorname{Crit} (\mathcal{A}^{H_0}) \hookrightarrow \operatorname{Crit} (\mathcal{A}^{H}),
\end{equation}
 such that its restriction to $\Crit(\cA^{H_0})\big\backslash (\Sigma_0{\times}\{0\})$ is a diffeomorphism. By Lemma \ref{lem:orbits} and Proposition \ref{prop:muCZ} the inclusion $i \colon  \Crit(\cA^{H_0}) \hookrightarrow \Crit(\cA^H)$ is both action and degree preserving\footnote{with respect to the transverse Conley-Zehnder index.}.

By Remark \ref{rem:Sk} for each connected component $\Lambda \subseteq\operatorname{Crit}(\cA^H)\big\backslash (\Sigma{\times}\{0\})$ there exists $m\in \{1, \dots, k\}$, such that $\Lambda$ is diffeomorphic to $S^{2m-1}$.  By Remark \ref{H2}, we have that $\Sigma_0\times\{0\}\subseteq \Crit(\cA^{H_0})$ is diffeomorphic to $S^{2k-1}$, while $\Sigma\times\{0\}\subseteq \Crit(\cA^H)$ is diffeomorphic to $S^{n+k-1}{\times}\R^{n-k}$. Moreover this diffeomorphism can be chosen such that if $ \Sigma_1 \subset \Sigma$ is the sphere of dimension $n+k-1$ corresponding to $ S^{n+k-1} \times \{0\}$ then $\Sigma_0 \subset \Sigma_1$. As in \eqref{eqj} we denote by $j \colon \Sigma_0 \hookrightarrow \Sigma_1$ the inclusion. Thus the inclusion $i$ from \eqref{eqi} satisfies
\[
i(x,0) = (j(x),0), \qquad \forall (x,0) \in \Sigma_0 \times \{0 \}.
\]
Therefore, we can choose a Morse-Smale pair $(f,g)$ on $\Crit (\cA^H)$, such that:
\begin{enumerate}[label=\roman*)]
\item $f$ is coercive;
\item  $\operatorname{Crit}(f)\cap\Lambda=\{z^-,z^+\}$ for each connected component $\Lambda\subset\Crit (\cA^H)$;
\item \label{cond4}
$(f_0,g_0)$ is a Morse-Smale pair on  $\Crit(\cA^{H_0})$, such that $(f_0,g_0):=(f \circ i, i^*g)$ on $\Crit(\cA^{H_0})\big\backslash(\Sigma_0{\times}\{0\})$;
\item for $\mathcal{A}^H(z^\pm)=0$ we have $W^s_f(z^\pm) \subseteq \Sigma_1{\times}\{0\}$,  $z^+ \in i(\Sigma_0)$ and $z^- \notin i(\Sigma_0)$;
\item $\Crit(f_0)\cap (\Sigma_0 {\times} \{0\})=\{x^-,x^+\}$ and $i(x^+)=z^+$ and $i(x^-) \notin \Crit(f)$.
\end{enumerate}
We denote by $x^-$ or $z^-$ always the minimum of $f_0$ or $f$ on a connected component $\Lambda_0$ of $\operatorname{Crit}(\mathcal{A}^{H_0})$ or $\Lambda$ of $\operatorname{Crit}(\mathcal{A}^{H})$ respectively. In the following, let $x^\pm$ or $z^\pm$ be the two critical points of $f_0$ or $f$ belonging to the same components $\Lambda_0$ or $\Lambda$ respectively. From the assumptions above, we can conclude the following:
\begin{enumerate}[label=\alph*)]
\item $\operatorname{Crit}(f_0)\cap\Lambda_0=\{x^\pm\}$ for all connected components $\Lambda_0 \subset \operatorname{Crit}(\mathcal{A}^{H_0})$.
\item The restriction $i \colon  \Crit(f_0)\big\backslash(\Sigma_0 {\times} \{0\}) \to \Crit(f)\big\backslash(\Sigma {\times} \{0\})$ is a bijection with $i(x^\pm)=z^\pm$.
\item The signature index of $x^{\pm}$ on $\Lambda_0 \cong S^{2m-1}$ is
\begin{equation}\label{morse-a}
\mu_\sigma(x^-) = - m + \frac{1}{2}, \qquad \mu_{\sigma}(x^+) = m - \frac{1}{2}.
\end{equation}
\item For $\eta\neq 0$, the signature index of $z^{\pm}$ on $\Lambda \cong S^{2m-1}$ is given by
\begin{equation}\label{morse-b}
\mu_\sigma(z^-) = - m + \frac{1}{2}, \qquad \mu_{\sigma}(z^+) = m - \frac{1}{2}.
\end{equation}
\item For $\eta=0$, we have $z^\pm\in\Sigma_1{\times}\{0\}\subset\Lambda=\Sigma{\times}\{0\}$ and $z^+$ is the maximum of $f|_{\Sigma_1}$ on $\Sigma_1\cong S^{n+k-1}$. The signature index of $z^{\pm}$ in this case is given by
\begin{equation}\label{morse-c}
\mu_\sigma(z^-) = - n + \frac{1}{2}, \qquad \mu_{\sigma}(z^+) = k - \frac{1}{2}.
\end{equation}
\end{enumerate}
We will construct the isomorphism from Theorem \ref{thm:iso} via moduli spaces of cascades with solutions to the hybrid problem \eqref{Floer}-\eqref{coupling} defined as follows:

For a pair $(x,z)\in \Crit(f_0)\times \Crit(f)$ and $m\in\N$ we denote by $\cM^m_{\operatorname{hyb}}(x,z)$ the set consisting of sequences $\left(\{u_l\}_{l=1}^m,\{t_l\}_{l=1}^{m-1}\right)$, such that exactly one $u_{l_0}$ in the sequence is a solution to the hybrid problem, whereas all other $u_l$ are Floer trajectories of $\cA^{H_0}$ for $l<l_0$ or of $\cA^H$ for $l>l_0$. We require 
\begin{align*} 
\phi^{t_l}\circ \operatorname{ev}^+(u_l)=\operatorname{ev}^-(u_{l+1})\qquad l=1,...,m-1,\\
\ev^-(u_1)\in W^u_{f_0}(x) \qquad \textrm{and}\qquad \ev^+(u_m)\in W^s_f(z),
\end{align*}
where $t_l\geq 0$ are real numbers and $\phi^{t_l}$ is for $l<l_0$ the time $t_l$ gradient flow of $(f_0,g_0)$ on $\Crit(\cA^{H_0})$ and for $l\geq l_0$ the time $t_l$ gradient flow of $(f,g)$ on $\Crit(\cA^{H})$.

\begin{remark}\label{rem:Mmanifold}
For $m>1$ the group $\R^{m-1}$ acts by time shift on Floer trajectories of $\cM_{\operatorname{hyb}}^m(x,z)$. We consider the quotients of $\cM_{\operatorname{hyb}}^m(x,z)$ by this action and define
\[
\cM_{\operatorname{hyb}}(x,z)\coloneqq \cM_{\operatorname{hyb}}^1(x,z)\cup \bigcup_{m>1}\left(\cM_{\operatorname{hyb}}^m(x,z)\left/\R^{m-1}\right.\right).
\]
$\mathcal{M}_{\operatorname{hyb}}(x,z)$ carries the structure of a smooth manifold by Theorem \ref{thm-index} and Proposition \ref{prop:trans} together with the standard Floer-theoretical results (cf.\ \cite[Prop.\ 1b]{Floer1989}, \cite[Thm.\ 9.2.3]{Audin2014}, \cite[Sec.\ 2.4]{Fauck2015}, \cite[Cor.\ A.15]{Frauenfelder2004},  \cite[Thm.\ 4.2]{pasquotto2018}, \cite[Sec.\ 2.3]{AbbonMerry2014}). Its dimension is $\mu(z)-\mu(x)$.
\end{remark}
If $\mathcal{A}^{H_0}(x)=\eta=\mathcal{A}^H(z)$, $\Lambda\coloneqq \Lambda^\eta$ and $\Lambda_0\coloneqq \Lambda^\eta_0$, we can describe $\mathcal{M}_{\operatorname{hyb}}(x,z)$ more explicitly. As the action decreases along non-trivial Floer trajectories, we find that in this case $\mathcal{M}_{\operatorname{hyb}}^m(x,z)=\emptyset$ for $m>1$. Thus $\mathcal{M}_{\operatorname{hyb}}(x,z)=\mathcal{M}_{\operatorname{hyb}}^1(x,z)$ consists entirely of stationary solutions. More precisely, we find that $\mathcal{M}_{\operatorname{hyb}}(x,z)$ is given by the fibre product
\[
\xymatrix{
\mathcal{M}_{\operatorname{hyb}}(x,z) \ar@{.>}[r] \ar@{.>}[d] & \mathcal{M}_{\operatorname{hyb}}( \Lambda_0 ,  \Lambda) \ar[d]^{\operatorname{ev}} \\
W^u_{f_0}(x) \times W^s_f(z) \ar[r]^{\qquad \iota} & \Lambda_0 \times \Lambda}
\]
where $\iota$ is the inclusion and $\operatorname{ev} \colon \mathcal{M}_{\operatorname{hyb}}(\Lambda_0,\Lambda)\to\Lambda_0\times\Lambda$ the evaluation map:
\[\operatorname{ev}(u_0,u)=\big(\lim_{s\rightarrow-\infty} u_0(s),\lim_{s\rightarrow+\infty} u(s)\big).\]
Hence $\mathcal{M}_{\operatorname{hyb}}(x,z)=\operatorname{ev}^{-1}\big(W^u_{f_0}(x) \times W^s_f(z)\big)$ is a manifold of dimension
\begin{align}
\dim \mathcal{M}_{\operatorname{hyb}}(x,z) & =
\dim W^u_{f_0}(x)+\dim W^s_f(z)+\dim \mathcal{M}_{\operatorname{hyb}}( \Lambda_0 ,  \Lambda)-\dim\Lambda_0-\dim\Lambda\nonumber\\
& = \dim W^u_{f_0}(x)+\dim W^s_f(z)-\frac{1}{2}\left(\dim\Lambda_0+\dim\Lambda\right)\nonumber\\
& = \mu_\sigma(z)-\mu_\sigma(x).
\label{dim:Wxz}
\end{align}
provided the following intersection is transverse in $ \Lambda_0 \times \Lambda$:
\begin{equation}\label{eq:transWxz}
W^u_{f_0}(x) {\times} W^s_f(z)\cap {\operatorname{ev}}\big(\mathcal{M}_{\operatorname{hyb}}( \Lambda_0 ,  \Lambda)\big).
\end{equation}
For $\eta\neq 0$, we have $\Lambda_0\cong \Lambda$, so that ${\operatorname{ev}}\big(\mathcal{M}_{\operatorname{hyb}}( \Lambda_0 ,  \Lambda)\big)= \Delta$ is the diagonal in $\Lambda_0\times\Lambda_0$ by Lemma \ref{lem:stationary}. On the other hand, by our assumptions on $f_0$ and $f$ we have that $i(W^u_{f_0}(x))=W^u_f(i(x))$ and consequently the transversality of \eqref{eq:transWxz} is equivalent to the transversality of $W^u_f(i(x))\pitchfork W^s_f(z)$, which follows from the Morse-Smale assumption on the flow of $\nabla f$.

When $\eta = 0$, we have by Lemma \ref{lem:stationary} that
\[
Y \coloneqq \operatorname{ev}\big( \mathcal{M}_{\operatorname{hyb}}(\Lambda_0, \Lambda \big) \cong  \Delta \times \chi_0,
\]
where $ \Delta$ is the diagonal in $\Lambda_0{\times}\Lambda_0\subseteq\Lambda_0{\times}\Lambda$ and $\chi_0$ is the zero section in $T^\ast\R^{n-k}$. 
In particular the transversality of \eqref{eq:transWxz} is equivalent to the transversality of
\begin{equation}\label{transWxz2}
\sigma^{-1}\big(W^u_{f_0}(x){\times} \chi_0\big) \cap W^s_f(z),
\end{equation}
in $\Sigma$.
By Remark \ref{H2}, we have
\begin{align*}
\sigma^{-1}(\Sigma_0{\times} \chi_0) \cap \Sigma_1 & = \sigma^{-1}(\Sigma_0{\times} \chi_0) \cap \Sigma \cap \sigma^{-1}(T^*\R^k{\times} S)\\
& = \Sigma \cap \sigma^{-1}(\Sigma_0{\times} (\chi_0 \cap S))
&&\mspace{-20 mu}=\Sigma \cap \sigma^{-1}(\Sigma_0{\times} \{0\}) =i(\Sigma_0).
\end{align*}
By assumption $z^- \notin i(\Sigma_0)$ and $W^s_f(z^-)=\{z^-\}$, which in view of the above gives:
$$
\sigma^{-1}\big(W^u_{f_0}(x^\pm){\times} \chi_0\big) \cap W^s_f(z^-)=\emptyset\quad \textrm{and} \quad \left(W^u_{f_0}(x^\pm)\times W^s_f(z^-)\right)\cap Y=\emptyset.
$$
Thus \eqref{eq:transWxz} is transverse for $(x,z)=(x^\pm,z^-)$.

On the other hand, by assumption $W^s_f(z^+)= \Sigma_1\setminus \{z^-\}\supseteq i(\Sigma_0)$. Thus
$$
\sigma^{-1}\big(W^u_{f_0}(x^\pm){\times} \chi_0\big) \cap W^s_f(z^+)\supseteq i(W^u_{f_0}(x^\pm))\neq \emptyset.
$$
Take $w \in \sigma^{-1}\big(W^u_{f_0}(x^\pm){\times} \chi_0\big) \cap W^s_f(z^+)$ and denote $\sigma(w)=(w_1,w_2)$. Then
\begin{align*}
T_w W^s_f(z^+) = T_w\Sigma_1 & = T_w\Sigma\cap D\sigma^{-1}(T_{w_1}T^*\R^k \times T_{w_2}S),\\
T_w \sigma^{-1}\big(W^u_{f_0}(x^\pm){\times} \chi_0\big) & = D\sigma^{-1}\big(T_{w_1}W^u_{f_0}(x^\pm)\times T_{w_2}\chi_0\big).
\end{align*}
By Remark \ref{H2} $\dim S= n-k = \dim \chi_0$ and $S\cap \chi_0=\{0\}$, hence $\operatorname{span}\{S, \chi_0\}=\R^{2(n-k)}$ and consequently the intersection $\sigma^{-1}(W^u_{f_0}(x^\pm)\times \chi_0) \cap W^s_f(z^+)$ is transverse.

\begin{remark}\label{rem:Mxz=1}
In all the above cases for $\mathcal{A}^H(z)=\mathcal{A}^{H_0}(x)$ and $\mu_{\sigma}(z)=\mu_\sigma(x)$, we find that $\mathcal{M}_{\operatorname{hyb}}(x,z)=\{(x,z)\}$ contains exactly one element.
\end{remark} 

In order to use $\mathcal{M}_{\operatorname{hyb}}(x,z)$ for the definition of $\psi$, one has to show that it is compact modulo breaking. This follows mostly by standard techniques via Gromov compactness. However, as we are on a non-compact manifold, we need to assure that all moduli spaces $\mathcal{M}_{\operatorname{hyb}}(x,z)$ are uniformly bounded in the $L^\infty$-norm.
\begin{lem}\label{lem:compact}
For every pair $(x,z)\in \Crit(f_0)\times \Crit(f)$ all Floer and hybrid trajectories from the set $\cM_{\operatorname{hyb}}(x,z)$ are uniformly bounded in the $L^\infty$-norm and the bound depends only on $x$ and $y$.
\end{lem}
\begin{proof}
Let $a\coloneqq \cA^{H_0}(x)$ and $b\coloneqq \cA^H(z)$.
Denote the moduli spaces of Floer trajectories of $\cA^{H_0}$ and $\cA^H$ in the action window $[a,b]$ as follows:
\begin{align*}
\mathfrak{M}_{0}(a,b) & \coloneqq \bigcup_{\substack{\Lambda^{\pm}\subseteq \Crit\left(\cA^{H_{0}}\right)\\  \cA^{H_{0}}\left(\Lambda^{\pm}\right)\in [a,b]}}\mathcal{M}(\Lambda^{-},\Lambda^{+}),\allowdisplaybreaks\\
 \mathfrak{M}_{1}(a,b) & \coloneqq \bigcup_{\substack{\Lambda^{\pm}\subseteq \Crit\left(\cA^H\right)\\ \cA^H\left(\Lambda^{\pm}\right)\in [a,b]}}\mathcal{M}(\Lambda^{-},\Lambda^{+}).
\end{align*}
By Corollary \ref{cor:CritVal} we know that the set $\operatorname{CritVal}(\cA^H)\cap[a,b]=\operatorname{CritVal}(\cA^{H_0})\cap[a,b]$ is finite. Combined with \cite[Thm. 1]{pasquotto2017} we infer that $\mathfrak{M}_{0}(a,b)$ and $\mathfrak{M}_{1}(a,b)$ are finite unions of sets whose images in $\R^{2k+1}$ and $\R^{2n+1}$ respectively are bounded, thus their images are bounded.

Using the notation from section \ref{ssec:bounds} we denote for a compact subset $N\subseteq \Sigma$ the moduli spaces of the solutions to \eqref{Floer} and \eqref{coupling} connecting components of the critical set of $\cA^{H_{0}}$ in action window $[a,b]$ with components of the critical set of $\cA^H$ in the same action window as follows:
\[
\mathfrak{M}_{2}(a,b)\coloneqq  \bigcup_{\substack{\Lambda_0\subseteq \Crit\left(\cA^{H_{0}}\right)\\ \cA^{H_{0}}\left(\Lambda_0\right)\in [a,b]}}\ \bigcup_{\substack{\Lambda \subseteq \mathcal{C}\left(\cA^H,N\right)\\ \cA^{H}\left(\Lambda\right)\in [a,b]}}\mathcal{N}_{\operatorname{hyb}}(\Lambda_0 , \Lambda).
\]
Again, the set $\mathfrak{M}_{2}(a,b)$ is a finite union of sets, which by Proposition \ref{prop:hybBound} are all bounded in $L^\infty$-norm, thus its image is also bounded in $L^\infty$-norm and the bound depends only on $a, b$ and $N$.

Fix $m\in \N$  and take $\left(\{u_l\}_{l=1}^m,\{t_l\}_{l=1}^{m-1}\right)\in \cM_{\operatorname{hyb}}^m(x,z)$. Then by definition of $\cM_{\operatorname{hyb}}^m(x,z)$ there exists $l_0\in \{1,\dots, m\}$, such that $u_{l_0}=(u^-,u^+)$ is a solution to the hybrid problem \eqref{Floer} and \eqref{coupling}. By Lemma \ref{energy} we have that
\[
a\leq \cA^{H_0}\circ \ev^-(u^-)\leq \cA^H\circ \ev^+(u^+)\leq b.
\]
From the above inequality we can conclude that for all $l\in \{1,\dots, l_0-1\}$ the corresponding $u_l$ is a Floer trajectory of $\cA^{H_0}$, with 
\[
a\leq \cA^{H_0}\circ \ev^-(u_l)\leq \cA^{H_0}\circ \ev^+(u_l)\leq \cA^{H_0}\circ \ev^-(u^-)\leq b,
\]
and therefore $u_l\in \mathfrak{M}_{0}(a,b)$, whereas for all $l\in \{l_0+1,\dots, m\}$ we have $u_l\in \mathfrak{M}_{1}(a,b)$ as these $u_l$ are Floer trajectories of $\cA^H$ with 
\[
b\geq \cA^H\circ \ev^+(u_l) \geq \cA^H\circ \ev^-(u_l)\geq \cA^H\circ \ev^+(u^+)\geq a.
\]
Now we consider three cases, choosing the compact set $N$ appropriately to each case:
\begin{enumerate}[label=\arabic*.]
\item If $0\notin [a,b]$ then none of the cascades passes through $\Sigma$; thus we can choose $N=\emptyset$ and for $u_{l_0}$ to be in $\mathfrak{M}_{2}(a,b)$.
\item If $b=0$ then $\ev^+(u_m)\in \Sigma{\times}\{0\}$ and $\ev^+(u_m)\in W^s_f(z)$. In particular, $\ev^+(u_m)\in f^{-1}\big((-\infty,f(z)]\big)$. Therefore, for $b=0$ we take $N\coloneqq f^{-1}\big((-\infty,f(z)]\big)$. Due to coercivity of $f$, this $N$ is compact. That ensures $\ev^+(u_{l_0})\in \mathcal{C}\left(\cA^H,N\right)$ and $u_{l_0} \in \mathfrak{M}_{2}(a,b)$.
\item If $0 \in [a,b)$ then take $N\coloneqq K(b)$ to be the \emph{shade} (cf.\ \cite[Sec.\ 4.1]{pasquotto2018}), i.e.
\begin{align*}
\widetilde{K}(b) &\coloneqq \bigcup_{\substack{\Lambda \subseteq  \Crit({\cA^{H}}),\\ {\cA^{H}}(\Lambda)\in (0,b]}}\ev^{-}(\mathcal{M}(\Sigma\times\{0\},\Lambda)),\\
K(b) &\coloneqq  f^{-1}\Big(\big(-\infty,\,{\textstyle \sup_{\tilde{K}(b)}f}\big] \Big).
\end{align*}
By \cite[Lem.\ 4.1]{pasquotto2018} the set $K(b)$ is a compact subset of $\Sigma{\times}\{0\}$. Moreover, if $\ev^+(u_{l_0})\in \Sigma{\times}\{0\}$, then $\ev^-(u_{l_0+1})\in \widetilde{K}(b)$ and consequently $\ev^+(u^+)=\phi^{-t_{l_0}}\circ \ev^-(u_{l_0+1})\in K(b)$. Thus taking $N\coloneqq K(b)$ ensures $u_{l_0}\in \mathfrak{M}_{2}(a,b)$.
\end{enumerate}
In all the above cases we have chosen $N$ to be a compact set, such that $u_{l_0}\in \mathfrak{M}_{2}(a,b)$. Consequently all elements of $\cM_{\operatorname{hyb}}^m(x,z)$ are bounded in the $L^\infty$-norm and the bounds depend only on $x$ and $z$. 
\end{proof}

\paragraph{\textit{Proof of Theorem \ref{thm:iso}:}}
Let $H=H_0+H_1$ and $f$ and $f_0$ be as described at the beginning of this section. Let $CF_*(H_0,f_0)$ and $CF_*(H,f)$ be the chain complexes associated to the quadruples $(H_0,J_0,f_0,g_0)$ and $(H,J,f,g)$, respectively, as defined in Section \ref{section_setting}. Now we define the chain map $\psi \colon  CF_*(\cA^{H_0},f_0) \to CF_*(\cA^H,f)$.

 For a given pair of points $(x,z)\in \Crit(f_0)\times \Crit(f)$ such that $\mu(x)=\mu(z)$ we find by Remark \ref{rem:Mmanifold} that $\cM_{\operatorname{hyb}}(x,z)$ is a 0-dimensional manifold. It is compact by Lemma \ref{lem:compact}, and hence a finite set. Denote its parity by
\begin{equation}
n(x,z)\coloneqq \#_2 \cM_{\operatorname{hyb}}(x,z) \in \Z_2.
\label{n(x,z)}
\end{equation}
Define the homomorphism $\psi \colon CF_*(H,f)\to CF_*(H_0,f_0)$ as the linear extension of:
\[
\psi(z)\coloneqq \sum_{\substack{x\in \Crit(f_0),\\ \mu(x)=\mu(z)}}n(x,z)x.
\]
For $\psi$ to be well defined, it has to satisfy the Novikov finiteness condition \eqref{novikov}, i.e.\ that for all $z\in \Crit(f)$ and $a\in \R$ holds
\begin{equation}\label{Nov}
\#\left\lbrace x\in \Crit(f_0)\ \big|\ n(x,z)\neq 0 \quad \textrm{and}\quad \cA^{H_0}(x)\geq a\right\rbrace<+\infty.
\end{equation}
By Lemma \ref{energy} the condition $n(x,z)\neq 0$ implies $\cA^{H_0}(x)\leq \cA^H(z)$. In other words, using the notation from \eqref{CFt}, for every $t\in \R$
\begin{equation}\label{psiCF}
\psi\left( CF^{\leq t}_*\left(H,f\right)\right)\subseteq CF^{\leq t}_*\left(H_0,f_0\right).
\end{equation}
On the other hand, all connected components of $\Crit(\cA^{H_0})$ are compact and $f_0$ is Morse, hence $\Crit(f_0)\cap \left(\cA^{H_0}\right)^{-1}\left([a,\cA^H(z)]\right)$ is finite and \eqref{Nov} is satisfied.

Lemma \ref{lem:compact} and Theorem \ref{thm-index} together with the standard gluing and compactness arguments (cf.\ \cite[Thm.\ 4.2, Thm.\ 11.1.16]{Audin2014}, \cite[Thm.\ A.11]{Frauenfelder2004}, \cite[Sec.\ 4.2]{AbbonMerry2014}) imply that $\psi$ commutes with the respective boundary operators, i.e.\ $\psi\circ\partial=\partial_0\circ\psi$, and thus induces a homomorphism 
\[
\Psi \colon  RFH_*(H)\to RFH_*(H_0).
\]
Recall the short exact sequence (\ref{ActionShortExSeq}), which was induced by action filtration:
\[0\rightarrow CF^0(H,f)\rightarrow CF^{\geq 0}(H,f)\rightarrow CF^+(H,f)\rightarrow 0.\]
Note that there is an analogous sequence for $(H_0,f_0)$. As $\psi$ reduces action (cf.\ \eqref{psiCF}) and commutes with $\partial$ and $\partial_0$, we find that it induces maps on the filtered chain complexes which fit into the following commutative diagram of complexes:
\[\xymatrix{0 \ar[r] & CF^0_\ast(H,f)\ar[d]^{\psi^0} \ar[r] & CF^{\geq0}_\ast(H,f)\ar[d] \ar[r] & CF^+_\ast(H,f)\ar[d]^{\psi^+} \ar[r] & 0\\
0 \ar[r] & CF^0_\ast(H_0,f_0) \ar[r] & CF^{\geq0}_\ast(H_0,f_0) \ar[r] & CF^+_\ast(H_0,f_0) \ar[r] & 0.}\]
By naturality (cf.\ \cite[Sec.\ 2.1]{hatcher2000}), we hence obtain the following commutative diagram of long exact sequences in homology:
\[\xymatrix{
\longrightarrow H_{\ast+ n-1}(\Sigma) \ar[d]^{\Psi^0} \ar[r] & RFH^{\geq0}_\ast(H) \ar[d]\ar[r] & RFH^+_\ast(H) \ar[d]^{\Psi^+} \ar[r] & H_{\ast+n-2}(\Sigma) \ar[d]^{\Psi^0}\longrightarrow\\
\longrightarrow H_{\ast+ k-1}(\Sigma_0) \ar[r] & RFH^{\geq0}_\ast(H_0) \ar[r] & RFH^+_\ast(H_0) \ar[r] & H_{\ast+k-2}(\Sigma_0) \longrightarrow.}\]
It remains to show that $\Psi^+$ is an isomorphism and that $\Psi^0 \colon H_{n+k-1}(\Sigma)\rightarrow H_{2k-1}(\Sigma_0)$ is an isomorphism and vanishes otherwise:

Recall that by Lemma \ref{lem:orbits} we have a natural bijection between $\Crit^+(\cA^H)$ and $\Crit^+(\cA^{H_0})$. Moreover, we defined  $f_0$ to be the pullback of $f$ under this bijection, such that their critical points are in $1$-to-$1$ correspondence, i.e.\ $(v,\eta)\in \Crit^+(f_0)$ if and only if $(\sigma^{-1}(v,0),\eta)\in \Crit^+(f)$. This correspondence allows us to represent the homomorphism $\psi^+$ as an infinite matrix with entries 
\[
n(x_1,i(x_2))\coloneqq \begin{cases}\#_2 \cM_{\operatorname{hyb}}(x_1,i(x_2)) & \text{if }\mu(x_1)=\mu(i(x_2))\\ 0 &\text{otherwise}\end{cases}
\] 
defined as in \eqref{n(x,z)} for $x_1,x_2\in \Crit(\cA^{H_0})$. In fact by \eqref{psiCF} this matrix is upper triangular. We would like to investigate its diagonal.
 
Fix $x\in \Crit(f_0)\cap \Crit^+(\cA^{H_0})$.
By Proposition \ref{prop:muCZ} we know that the inclusion $i \colon  \Crit(\cA^{H_0}) \hookrightarrow \Crit(\cA^H)$ preserves the Conley-Zehnder index, i.e.\ $\mu_{CZ}(x)=\mu_{CZ}(i(x))$. On the other hand, since $\Crit^+(\cA^{H_0})$ is diffeomorphic to $\Crit^+(\cA^H)$ with $f_0$ being the pullback of $f$ under this diffeomorphism, we infer that $\mu_\sigma(x)=\mu_\sigma(i(x))$. 
Consequently, $\mu(x)=\mu(i(x))$ and $n(x, i(x))=\#_2 \cM_{\operatorname{hyb}}(x,i(x)) =1$ by Remark \ref{rem:Mxz=1}.

We can conclude that the matrix representing $\psi^+$ is upper triangular with $1$'s on the diagonal and therefore $\psi^+$ is an isomorphism and induces on homology level the isomorphism $\Psi^+$.

For $\Psi^0$ recall from the beginning of this section, that we have four action zero critical points $\{x^\pm\}\in\Lambda^0_0\subseteq\Crit(\cA^{H_0})$ and $\{z^\pm\}\in\Lambda^0\subseteq \Crit(\cA^H)$. All four points correspond to constant orbits of constant flows. Hence their Conley-Zehnder index is zero. From the calculations of the signature indexes \eqref{morse-a} and \eqref{morse-c}, we infer that
\[
\mu(z^+)= \mu(x^+)=k,\quad \mu(z^-)=-n+1,\quad \mu(x^-)=-k+1.
\]
Thus, we find that $n(x^+,z^+)=\#_2\mathcal{M}_{\operatorname{hyb}}(x^+,i(x^+))=1$ by Remark \ref{rem:Mxz=1} and that for any other pair $(x,z)\in \{(x^-,z^-),(x^-,z^+),(x^+,z^-)\}$ we have
$n(x,z)=0$, as $\mu(x)\neq\mu(z)$. Hence, $\Psi^0_k \colon  H_{n+k-1}(\Sigma)\rightarrow H_{2k-1}(\Sigma_0)$ is an isomorphism and $\Psi^0_\ast=0$ for $\ast\neq k$.

Finally, recall from \eqref{umkehr} that the map $H_*( \Sigma) \to H_{*+k-n}( \Sigma_0)$ given by composing the Umkehr map of the inclusion $ j \colon \Sigma_0 \hookrightarrow \Sigma_1$ together with a retraction $r \colon \Sigma \to \Sigma_1$ is an isomorphism for $\ast= n+k-1$ and $0$ otherwise. Therefore up to a degree shift coming from the difference between the signature index grading and the Morse grading, $\Psi^0$ agrees\footnote{This can also be seen directly from definition of $ \psi^0$, using the Morse-theoretic description of the Umkehr map from \cite[App. B]{Abbon2009} or \cite[App.\ A, p1716-1717]{AbbondandoloSchwarz2010}.} with the $j_! \circ r_*$. \hfill $\square$

\begin{remark}\label{rem:Sigma-}
Analogously, we can define $\Psi:RFH^-(H)\rightarrow RFH^-(H_0)$ and show that it gives an isomorphism.
\end{remark}
\subsection{Computation}
In this section we prove Theorem \ref{thm:compRFH}. In other words, we will calculate the full Rabinowitz Floer homology of $H$ using the isomorphism of the positive Rabinowitz Floer homologies proven in Theorem \ref{thm:iso} and the following long exact sequences (see \eqref{ActionLongExSeq}):
\begin{gather}
\dots \to H_{*+n-1}(\Sigma) \to RFH_{*}^{\geq 0}(H) \to RFH^+_*(H)\to \dots \label{exact1}\\
\dots \to RFH^-_{*}(H) \to RFH_{*}(H) \to RFH^{\geq 0}_*(H)\to \dots \label{exact2}
\end{gather}

We begin by collecting the properties of the Rabinowitz Floer homology for the compact hypersurface $ \Sigma_0$ that we need\footnote{ We remind the reader (Remark \ref{rem:onlyonsigma}) that in the compact case the Rabinowitz Floer homology only depends on $H_0$ through its zero level set $ \Sigma_0$. Therefore instead of $RFH_*(H_0)$ we could write $RFH_*(\Sigma_0)$, and similarly for the other variants. For consistency with the non-compact case, however, we will not do this.}. 
\begin{prop}
\label{compact}
For the compact hypersurface $ \Sigma_0 \subseteq \mathbb{R}^{2k}$ we have
\begin{equation}\label{plusIminus1}
RFH^+_*(H_0) =  \left\lbrace
\begin{array}{c c c}
\Z_2 &  & *= k+1,\\
0 & & *\neq  k+1.
\end{array}
\right.
\qquad
RFH^-_*(H_0)= \left\lbrace
\begin{array}{c c c}
\Z_2 & & *=  -k,\\
0 & & *\neq  -k.
\end{array}
\right.
\end{equation}
\end{prop}
\begin{proof}
Since $\Sigma_0$ is compact, it is displaceable in $\mathbb{R}^{2k}$. Thus by \cite[Thm.\ 1.2]{CieliebakFrauenfelder2009} we have $RFH(H_0)=0$. Thus also the symplectic homology $SH_*(\Sigma_0)$ and cohomology $SH^*(\Sigma_0)$ vanish by \cite[Thm.\ 13.3]{ritter_2013}. Therefore the positive symplectic homology $SH^+_*(\Sigma_0)$ agrees\footnote{This can also easily be proved directly.} with $H^{k-1-*}( \mathbb{R}^{2k})$ by \cite[Lem.\ 2.1]{CieliebakFrauenfelderOancea2010}. Then by \cite[Thm.\ 1.4]{CieliebakFrauenfelderOancea2010} we have $RFH^+_*(H_0) \cong SH^+_*( \Sigma_0)$. Now \eqref{plusIminus1} follows from \eqref{exact1} and \eqref{exact2}.
\end{proof}

That, together with the isomorphism from Theorem \ref{thm:iso} will allow us to calculate the full Rabinowitz Floer homology of $\Sigma$:

\vspace*{0.5cm}
\paragraph{\textit{Proof of Theorem \ref{thm:compRFH}:}}
Recall from Remark \ref{H2} that $\Sigma \simeq S^{n+k-1}\times\R^{n-k}$, hence
\begin{equation}\label{homSigma}
H_*(\Sigma) =  \begin{cases} \Z_2 & *=0, n+k-1,\\ 0 &\textrm{otherwise.}\end{cases}\quad\Rightarrow\quad H_{*+n-1}(\Sigma)=\begin{cases} \Z_2 & *=1-n, k,\\ 0 &\textrm{otherwise.}\end{cases}
\end{equation}
We now use the natural map (all squares commute!) between the long exact sequences proven in Theorem \ref{thm:iso}:
$$
\xymatrix{
 \dots \ar[r] &H_{*+n-1}(\Sigma) \ar[r] \ar[d]_{\Psi^0}
                 & RFH_{*}^{\geq 0}(H) \ar[r] \ar[d] &    RFH^+_*(H) \ar[r] \ar[d]_{\Psi^+} & \dots \\
\dots \ar[r]  & H_{*+k-1}(\Sigma_0)  \ar[r]   & RFH_{*}^{\geq 0}(H_0) \ar[r] & RFH^+_*(H_0) \ar[r]  & \dots}
$$
By Theorem \ref{thm:iso} the map $\Psi^+ \colon  RFH^+_*(H) \to  RFH^+_*(H_0)$ is an isomorphism. Hence, we obtain from \eqref{plusIminus1} and \eqref{homSigma} that $RFH_{*}^{\geq 0}(H)= 0$ for all $* \notin \{1-n, k+1, k\}$. In case $*= 1-n$ we have
\[
\equalto{RFH^+_{ 2-n}(H)}{0}\to \equalto{H_{0}(\Sigma)}{\Z_2} \to RFH_{ 1-n}^{\geq 0}(H) \to \equalto{RFH^+_{ 1-n}(H)}{0},
\]
which implies $RFH_{ 1-n}^{\geq 0}(H)=\Z_2$. In case $* \in \{ k+1, k\}$ we have 
$$
\xymatrix{
 0 \rightarrow RFH_{k +1}^{\geq 0}(H) \ar[r] \ar[d] &  RFH^+_{k  +1}(H) \ar[r] \ar[d]_{\Psi^+} &H_{n+k-1}(\Sigma) \ar[r] \ar[d]_{\Psi^0}
                 & RFH_{k}^{\geq 0}(H) \ar[d] \rightarrow 0\\
0 \rightarrow \equalto{RFH_{k +1}^{\geq 0}(H_0)}{0} \ar[r] &\equalto{RFH^+_{k  +1}(H_0)}{\Z_2} \ar[r] & \equalto{H_{2k-1}(\Sigma_0) }{\Z_2} \ar[r]   & \equalto{RFH_k^{\geq 0}(H_0)}{0} \rightarrow 0}
$$
Observe that the maps $RFH^+_{k  +1}(H_0)\rightarrow H_{2k-1}(\Sigma_0), \Psi^+$ and $\Psi^0$ are isomorphisms (cf.\ Theorem \ref{thm:iso}), implying that the map $RFH^+_{k  +1}(H)\rightarrow H_{n+k-1}(\Sigma)$ is also an isomorphism. Consequently, $RFH_*^{\geq 0}(H)=0$ for $* \in \{ k+1, k\}$. We conclude that
\[
RFH_*^{\geq 0}(H) \coloneqq  \left\lbrace
\begin{array}{c c l }
\Z_2 & & *= 1-n,\\
0 & & \textrm{otherwise}.
\end{array}
\right.
\]
Having determined $RFH^{\geq 0}_*(H)$ we plug it into \eqref{exact2} to determine $RFH(H)$. Combining it with \eqref{plusIminus1} and  Remark \ref{rem:Sigma-} we obtain
$RFH_\ast(H)=0$ for $*\notin \{1{-}n, -k\}$. In the other two cases we have the following exact sequences, which depend on the relation between $n$ and $k$. In case $n\neq k{+}1$ we have
\[
\equalto{RFH^-_{ 1-n}(H)}{0} \to RFH_{ 1-n}(H) \to \equalto{RFH^{\geq 0}_{ 1-n}(H)}{\Z_2}\to 0,
\]
and
\[
0 \to \equalto{RFH^-_{-k}(H)}{\Z_2} \to RFH_{-k}(H) \to \equalto{RFH^{\geq 0}_{-k}(H)}{0},
\]
which gives $RFH_{-k}(H)=RFH_{1-n}(H)=\Z_2$. However, if $n=k{+}1$ we have
\[
0 \to \equalto{RFH^-_{ 1-n}(H)}{\Z_2} \to RFH_{1-n}(H) \to \equalto{RFH^{\geq 0}_{ 1-n}(H)}{\Z_2}\to 0,
\]
which gives $RFH_{ 1-n}(H)=\Z_2\oplus \Z_2$. \hfill $\square$
\printbibliography
\end{document}